\documentclass{article}
\usepackage[utf8]{inputenc}
\usepackage{babel}
\usepackage{amsmath,amssymb}
\usepackage{geometry}
\usepackage{graphicx}
\usepackage{caption}
\usepackage{subcaption}
\usepackage{stmaryrd}
\usepackage{amsthm}
\usepackage{thmtools}
\usepackage{thm-restate}
\newtheorem{theorem}{Theorem}[section]
\newtheorem{lemma}{Lemma}[section]
\newtheorem{prop}{Proposition}[section]
\newtheorem{coro}{Corollary}[section]

\usepackage{ bbold }
\usepackage{dsfont}
\usepackage{csquotes}
\usepackage{hyperref}
\usepackage{fancyhdr}
\numberwithin{equation}{section}
\newcommand{\spa}[1]{\mathcal{#1}}
\newcommand{\op}{\pi}
\newcommand{\cov}{\Sigma}
\newcommand{\m}{\mu}
\newcommand{\n}{\nu}
\newcommand{\reel}{\mathbb{R}}
\newcommand{\expe}{\mathbb{E}}
\newcommand{\mroot}[1]{#1^{\frac{1}{2}}}
\newcommand{\imroot}[1]{#1^{-\frac{1}{2}}}

\title{Gromov-Wasserstein Distances between Gaussian Distributions}

\providecommand{\keywords}[1]
{
  \small	
  \textbf{\textit{Keywords---}} #1
}

\author{Antoine Salmona$^1$,  Julie Delon$^{2}$, Agnès Desolneux$^1$ \thanks{The authors acknowledge support from the French Research Agency through the PostProdLEAP project (ANR-19-CE23-0027-01) and the MISTIC project (ANR-19-CE40-005).}}
\date{%
    $^1$ ENS Paris-Saclay, CNRS, Centre Borelli UMR 9010\\%
    $^2$ Universit\'e de Paris, CNRS, MAP5 UMR 8145 and Institut Universitaire de France \\[2ex]%
    \today
}

\begin{document}
\maketitle
\begin{abstract}
\noindent The Gromov-Wasserstein distances were proposed a few years ago to compare distributions which do not lie in the same space. In particular, they offer an interesting alternative to the Wasserstein distances for comparing probability measures  living on Euclidean spaces of different dimensions. In this paper, we focus on the Gromov-Wasserstein distance with a ground cost defined as the squared Euclidean distance and we study the form of the optimal plan between Gaussian distributions. We show that when the optimal plan is  restricted to Gaussian distributions, the problem has a very simple linear solution, which is also solution of the linear Gromov-Monge problem. 
We also study the problem without restriction on the optimal plan, and provide lower and upper bounds for the value of the Gromov-Wasserstein distance between Gaussian distributions.    
\newline
\newline
\end{abstract}

\section{Introduction}
Optimal transport (OT) theory has become nowadays a major tool to compare probability distri\-butions.
It has been increasingly used over the last past years in various applied fields such as  economy~\cite{galichon2014stochastic}, image processing~\cite{color,textures},  machine learning~\cite{bigot2017geodesic,blanchet2019robust} or more generally data science \cite{OT}, with applications to domain adaptation \cite{courty2016optimal} or generative models \cite{arjovsky2017wasserstein,genevay2017learning}, to name just a few. 

Given two probability distributions $ \mu $ and $ \nu $ on two Polish spaces $ (\spa{X},d_{\spa{X}}) $ and $ (\spa{Y},d_{\spa{Y}}) $ and a positive lower semi-continuous cost function 
$ c: \mathcal{X} \times \mathcal{Y} \shortrightarrow \mathbb{R}^+ $,  optimal transport focuses on solving the following optimization problem
\begin{equation}\label{eq:generalOT}
\inf_{\op \in \Pi(\m,\n) } \int_{\spa{X} \times \spa{Y} } c(x,y)d\op(x,y),
\end{equation}
where $ \Pi(\m,\n)  $ is the set of measures on $ \spa{X} \times \spa{Y} $ with marginals $ \m $ and $ \n $. 
When $ \spa{X} $ and $ \spa{Y} $ are equal and Euclidean, typically $ \reel^d $, and $ c(x,y) = \| x - y \|^p $ with $ p \geq 1 $, 
Equation \eqref{eq:generalOT} induces a distance over the set of measures with finite moment of order $ p $, known as the $p$-Wasserstein distance $ W_p $:
\begin{equation}
W_p(\m,\n) = \left(\inf_{\op \in \Pi(\m,\n) } \int_{\reel^d \times \reel^d} \|x - y \|^pd\op(x,y)\right)^{\frac{1}{p}},
\end{equation}
or equivalently
\begin{equation}
W^p_p(\m,\n) = \inf_{X \sim \m, Y \sim \n} \expe[\|X-Y\|^p],
\end{equation}
where the notation $ X \sim \m $ means that $ X $ is a random variable with probability distribution $ \m $. 
It is known that Equation \eqref{eq:generalOT} always admits a solution \cite{reftheo1,reftheo2,reftheo5} , 
i.e. the infinimum is always reached. Moreover, in the case of  $ W_2 $, it is known~\cite{brenier} that if $ \m $ is absolutely continuous, 
then the \emph{optimal transport plan} $ \op^* $ is unique and has the form $ \op^* = (Id,T)\#\m $ where $ \# $ is the push-forward operator 
and $ T: \reel^d \shortrightarrow \reel^d $ is an application called \emph{optimal transport map}, satisfying $ T\#\m = \n $. 
The 2-Wasserstein distance $ W_2 $ admits a closed-form expression~\cite{dowson1982frechet, takatsu2010wasserstein}  when $ \m = \mathcal{N}(m_0,\cov_0) $ and $ \n = \mathcal{N}(m_1,\cov_1) $ are two Gaussian measures with means
$ m_0 \in \reel^d $, $ m_1 \in \reel^d $ and covariance matrices $ \cov_0 \in \reel^{d \times d} $ and $ \cov_1 \in \reel^{d \times d} $, that is given by
\begin{equation}\label{eq:w2gaussian}
W_2^2(\mu,\nu) = \|m_1 - m_0 \|^2 + \text{tr}\left(\cov_0 + \cov_1 - 2\mroot{\left(\mroot{\cov_0}\cov_1\mroot{\cov_0}\right)}\right),
\end{equation} 
where for any symmetric semi-definite positive $M $, $ \mroot{M} $ is the unique symmetric semi-definite positive squared root of $ M $. 
Moreover, if $ \cov_0 $ is non-singular, then the optimal transport map $ T $ is affine and is given by
\begin{equation}\label{eq:mongemapgaussian}
\forall x \in \reel^d, \ T(x) = m_1 + \imroot{\cov_0}\mroot{\left(\mroot{\cov_0}\cov_1\mroot{\cov_0}\right)}\imroot{\cov_0}(x-m_0),
\end{equation}
and the corresponding optimal transport plan $ \op^*$ is a degenerate Gaussian measure. 
 
For some applications such as shape matching or word embedding, an important limitation of classic OT lies 
in the fact that it is not invariant to rotations and translations and more generally to \emph{isometries}. 
Moreover, OT implies that we can define a relevant cost function to compare spaces $ \spa{X} $  and $ \spa{Y} $. Thus, when for instance $ \m $ is a measure on $ \reel^2 $ and $ \n $ a measure on $  \reel^3 $, 
it is not straightforward to design a cost function $ c: \reel^2 \times \reel^3 \shortrightarrow \reel $ and so one cannot define easily an OT distance 
to compare $ \m $ with $ \n $.  To overcome these limitations, several extensions of OT have been proposed \cite{invariance,dimension,temd}. 
Among them, the most famous one is probably the Gromov-Wasserstein  (GW) problem \cite{memoli}: given two Polish spaces $ (\spa{X},d_{\spa{X}}) $ 
and $ (\spa{Y},d_{\spa{Y}}) $, each endowed respectively with probability measures $\m $ and $ \n $,  
and given two measurable functions $ c_{\spa{X}} : \spa{X} \times \spa{X} \shortrightarrow \reel $ and  $ c_{\spa{Y}} : \spa{Y} \times \spa{Y} \shortrightarrow \reel $ , it aims at finding
\begin{equation}\label{eq:generalgromov}
GW_p(c_{\spa{X}},c_{\spa{Y}},\m,\n) =  \left( \inf_{\op \in \Pi(\m,\n)}  \int_{\spa{X}^2 \times \spa{Y}^2 } | c_{\spa{X}}(x,x') - c_{\spa{Y}}(y,y')|^p d\op(x,y)d\op(x',y') \right)^{\frac{1}{p}},
\end{equation}
with $p \geq 1$. As for classic OT, it can be shown that Equation \eqref{eq:generalgromov} always admits a solution (see \cite{These}).
The GW problem can be seen as a quadratic optimization problem in $ \op $, as opposed to OT, which is a linear optimization problem in $ \op $.
It induces a distance over the space of \emph{metric measure spaces} (i.e. the triplets $ (\spa{X}, d_{\spa{X}}, \m) $) quotiented by 
the \emph{strong isomorphisms} \cite{OT} \footnote{We say that $ (\spa{X}, d_{\spa{X}}, \m) $ is strongly isomorphic to $ (\spa{Y}, d_{\spa{Y}}, \n) $
if it exists a bijection $ \phi: \spa{X} \shortrightarrow \spa{Y} $ such that $ \phi $ is an isometry
($d_{\spa{Y}}(\phi(y),\phi(y')) = d_{\spa{X}}(x,x') $), and $ \phi \# \m = \n $.}. The fundamental metric properties of $GW_p $ have been studied 
in depth in \cite{reftheo3,memoli,reftheo4}. In the Euclidean setting, when $ \spa{X} = \reel^m $, $ \spa{Y} = \reel^n $, with $ m $ not
necessarily being equal to $ n $, and for the natural choice of costs $ c_{\spa{X}} = \|.\|^2_{\reel^m} $ and $ c_{\spa{Y}} = \|.\|^2_{\reel^n}  $, 
where $ \|.\|_{\reel^m} $ means the Euclidean norm on $ \reel^m $, it can be easily shown that $ GW_2(c_{\spa{X}},c_{\spa{Y}},\m,\n) $ 
is invariant to isometries. With a slight abuse of notations, we will note in the following $ GW_2(\m,\n) $ instead of $ GW_2(\|.\|_{\reel^m}^2,\|.\|_{\reel^n}^2,\m,\n)$. 

In this work, we focus on the problem of Gromov-Wasserstein between Gaussian measures. Given $ \m = \mathcal{N}(m_0,\cov_0) $,
 with $ m_0 \in \reel^m $ and with covariance matrix $ \cov_0 \in \reel^{m \times m}  $, and $ \n = \mathcal{N}(m_1,\cov_1) $, 
 with $ m_1 \in \reel^n $ and with covariance matrix $ \cov_1 \in \reel^{n \times n} $, we aim to solve
\begin{equation}\label{eq:gromovgeneral1}\tag{GW}
GW_2^2(\m,\n) = \inf_{\op \in \Pi(\m,\n)} \int \int \left(\|x - x'\|^2_{\reel^m} - \|y - y'\|^2_{\reel^n}\right)^2d\op(x,y)d\op(x',y'),
\end{equation}
or equivalently
\begin{equation}\label{eq:gromovgeneral}
GW_2^2(\m,\n) = \inf_{X,X',Y,Y' \sim  \op \otimes \op} \expe \left[\left(\|X - X'\|_{\reel^m}^2 - \|Y - Y'\|_{\reel^n}^2\right)^2 \right],
\end{equation}
where for $ x,x' \in \reel^m $ and $ y,y' \in \reel^n $, $ (\op \otimes \op)(x,x',y,y') = \op(x,y)\op(x',y') $. In particular, can we find equivalent formulas  to \eqref{eq:w2gaussian} and \eqref{eq:mongemapgaussian} in the case of Gromov-Wasserstein?
In Section~\ref{sec:2}, we derive an equivalent formulation of the Gromov-Wasserstein problem. This formulation is not specific to Gaussian measures 
but to all measures with finite order $4$ moment. It takes the form of a sum of two terms depending respectively on co-moments of 
order $2$ and $4$ of $ \op $. Then in Section \ref{sec:3}, we derive a lower bound by simply optimizing both terms separately. In Section \ref{sec:4}, we show that the 
problem restricted to Gaussian optimal plans admits an explicit solution and this solution is closely related to Principal Components Analysis (PCA). In Section \ref{sec:5}, we study the tightness of
the bounds found in the previous sections and we exhibit a particular case where we are able to compute exactly the value of $ GW_2^2(\m,\n) $ and the optimal plan $ \op^* $ which achieves it. Finally,  Section~\ref{sec:6} discusses the form of the solution in the general case, and the possibility that the optimal plan between two Gaussian distributions is always Gaussian. 

\section*{Notations}
We define in the following some of the notations that will be used in the paper. 
\begin{itemize}
\item The notation $  Y \sim \m $ means that $ Y $ is a random variable with probability distribution $ \m $.
\item If $ \m $ is a positive measure on $ \spa{X} $ and $ T:\spa{X} \rightarrow \spa{Y} $ is an application,
$ T\#\m $ stands for the push-forward measure of $ \m $ by $ T$, i.e. the measure on $ \spa{Y} $ such that  $ \forall A \in \spa{Y}$, $(T\#\m)(A) =\m(T^{-1}(A)) $.
\item If $ X $ and $ Y $  are random vectors on $ \reel^m $ and $ \reel^n $, we denote 
$ \text{Cov}(X,Y) $ the matrix of size $ m \times n $ of the form  $  \expe \left[(X - \expe[X])(Y - \expe[Y])^T \right] $.
\item the notation $ \text{tr}(M) $ denotes the trace of a matrix $ M $. 
\item $ \|M\|_{\spa{F}} $ stands for the Frobenius norm of a matrix $ M $, i.e. $ \|M\|_{\spa{F}} = \sqrt{\text{tr}(M^TM)} $. 
\item $ \text{rk}(M) $ stands for the rank of a matrix $M $.
\item $ I_n $ is the identity matrix of size $ n $. 
\item $\tilde{I}_n $ stands for any matrix of size $ n $ of the form $ \textup{diag}((\pm 1)_{i\leq n}) $
\item Suppose $ n \leq m $. For $ A \in \reel^{m \times m}$, we denote $ A^{(n)} \in \reel^{n \times n} $ the submatrix containing the $ n $ first rows and
the $ n $ first columns of $ A $.
\item Suppose $ n \leq m $. For $ A \in \reel^{n \times n}$, we denote $ A^{[m]} \in \reel^{m \times m} $ the matrix of the form $  \begin{pmatrix} A & 0 \\ 0 & 0 \end{pmatrix} $.
\item We denote $ S_n(\mathbb{R}) $ the set of symmetric matrices of size $ n $, $ S^+_n(\mathbb{R}) $ 
the set of  semi-definite  positive matrices, and $ S^{++}_n(\mathbb{R}) $ the set of definite positive matrices. 
\item $ \mathbb{1}_{n,m} = (1)_{i \leq n, j \leq m} $ denotes the matrix of ones with $ n $ rows and $ m $ columns. 
\item $ \|x\|_{\reel^n} $ stands for the Euclidean norm of $ x \in \reel^n $. We will denote $ \|x\|$ when there is no ambiguity about the dimension.
\item $ \langle x,x' \rangle_n $ stands for the Euclidean inner product in $ \reel^n $ between $ x $ and $ x' $.
\end{itemize} 

\section{Derivation of the general problem}\label{sec:2}
In this section, we derive an equivalent \footnote{We say that two optimization problems are equivalent if the solutions of one are readily obtained from the solutions of the other, and vice-versa.} formulation of problem  \eqref{eq:gromovgeneral1} which takes the form of a functional of co-moments of order $ 2 $ 
and $4$ of $ \op $. This formulation is not specific to Gaussian measures but to all measures with finite $ 4^{\text{th}} $ order moment. 

\begin{theorem}\label{thm:gromovcov}
Let $ \m $ be a probability measure on $ \reel^m $ with mean vector $ m_0  \in \reel^m $ and covariance matrix $ \cov_0 \in \reel^{m \times m} $ 
such that $ \int \|x\|^4d\m < + \infty $ and $ \n $ a probability measure on $ \reel^n $ with mean vector $ m_1 $ and 
covariance matrix $ \cov_1 \in \reel^{n \times n} $ such that $ \int \|y\|^4d\n < + \infty $. Let $ P_0, D_0 $  and $ P_1 , D_1 $ be 
respective diagonalizations of $ \cov_0 (= P_0D_0P_0^T)$ and $ \cov_1 (=P_1D_1P_1^T) $.
Let us define $ T_0 : x \in \reel^m \mapsto P_0^T(x - m_0)  $ and $ T_1 : y \in \reel^n \mapsto P_1^T(y - m_1)$. 
Then problem \eqref{eq:gromovgeneral1} is equivalent to problem
\begin{equation}\label{eq:gromovcov}\tag{supCOV}
\sup_{X \sim T_0\#\m, Y \sim T_1\#\n} \sum\limits_{i,j}\textup{Cov}(X_i^2,Y_j^2) + 2\|\textup{Cov}(X,Y)\|^2_{\mathcal{F}},
\end{equation}
where  $ X =  (X_1,X_2,\dots,X_m)^T $, $ Y = (Y_1,Y_2,\dots,Y_n)^T $, and $ \|.\|_{\mathcal{F}} $ is the Frobenius norm.
\end{theorem}

This theorem is a direct consequence of the two following intermediary results.
 
\begin{restatable}{lemma}{isometries}\label{lemme:isometries}
We denote $ \mathbb{O}_m = \{O \in \reel^{m \times m} |  \ O^TO = I_m \} $ the set of orthogonal matrices of size $ m $. 
Let $ \m $ and $ \n $ be two probability measures on $ \reel^m $ and $ \reel^n $. Let $ T_m : x \mapsto O_mx + x_m $ 
and $ T_n: y \mapsto O_ny + y_n $ be two affine applications with $ x_m \in \reel^m $, $ O_m \in \mathbb{O}_m $, $ y_n \in \reel^n$, 
and $ O_n \in \mathbb{O}_n $. Then $ GW_2(T_m\#\m,T_n\#\n) = GW_2(\m,\n) $.
\end{restatable}

\begin{restatable}[\textbf{\textup{Vayer, 2020, \cite{These}}}]{lemma}{vayer}
\label{lemme:vayer}
Suppose there exist some scalars $ a, b, c $ such that $ c_{\spa{X}}(x,x') = a\|x\|_{\reel^m}^2 + b\|x'\|_{\reel^m}^2 + c\langle x, x' \rangle_m  $,
where $ \langle .,. \rangle_m $ denotes the inner product on $ \reel^m $, and 
$ c_{\spa{Y}}(y,y') = a\|y\|_{\reel^n}^2 + b\|y'\|_{\reel^n}^2 + c\langle y, y' \rangle_n $. 
Let $\m$ and $\n$ be two probability measures respectively on $ \reel^m $ and $ \reel^n $.
Then
\begin{equation}\label{eq:vayer}
GW^2_2(c_{\spa{X}},c_{\spa{Y}},\m,\n) = C_{\m,\n} - 2\sup_{\op \in \Pi(\m,\n)} Z(\op),
\end{equation}
where $ C_{\m,\n} = \int c^2_{\spa{X}}d\m d\m + \int c^2_{\spa{Y}}d\n d\n - 4ab\int\|x\|_{\reel^m}^2\|y\|_{\reel^n}^2d\m d\n $ and
\begin{equation}\label{eq:Z}
\begin{split}
Z(\op) &= (a^2 + b^2)\int\|x\|_{\reel^m}^2\|y\|_{\reel^n}^2d\op(x,y) + c^2 \left\| \int xy^Td\op(x,y) \right\|_{\mathcal{F}}^2 \\
&+ (a +b)c\int \left( \|x\|_{\reel^m}^2\langle \mathbb{E}_{Y \sim \n}[Y],y \rangle_n + \|y\|_{\reel^n}^2\langle \mathbb{E}_{X \sim \m}[X],x \rangle_m \right)d\op(x,y).
\end{split}
\end{equation}
\end{restatable}

\begin{proof}[Proof of theorem \ref{thm:gromovcov}] Using Lemma \ref{lemme:isometries}, we can focus without any loss of generality on centered Gaussian measures with
diagonal covariance matrices. Thus, defining $ T_0 : x \in \reel^m \mapsto P_0^T(x-m_0) $ and $T_1 : y \in \reel^n \mapsto P_1^T(y - m_1) $
 and then applying Lemma \ref{lemme:vayer} on $ GW_2(T_0\#\m,T_1\#\n) $ with $a = 1$, $b = 1$, and $c = 2$ while remarking
that the last term in Equation \eqref{eq:Z} is null because $\expe_{X \sim T_0\#\m}[X] = 0$ and $\expe_{Y \sim T_1\#\n}[Y] = 0$, it comes
that problem \eqref{eq:gromovgeneral1} is equivalent to

\begin{equation}
\sup_{\op \in \Pi(T_0\#\m,T_1\#\n)} \int\|x\|_{\reel^m}^2\|y\|_{\reel^n}^2d\op(x,y) + 2 \left\| \int xy^Td\op(x,y) \right\|_{\mathcal{F}}^2.
\end{equation}

\noindent Since $ T_0\#\m $ and $ T_1\#\n $ are centered, we have that $ \int xy^Td\op(x,y)  = \text{Cov}(X,Y) $ where $ X \sim T_0\#\m $ 
and $ Y \sim T_1\#\n $. Furthermore, it can be easily computed that
\begin{equation}
\int\|x\|_{\reel^m}^2\|y\|_{\reel^n}^2d\op(x,y) = \sum\limits_{i,j} \text{Cov}(X_i^2,Y_j^2) + \sum\limits_{i,j}\expe[X_i^2]\expe[Y_j^2].
\end{equation}
Since the second term doesn't depend on $ \op $, we get that problem \eqref{eq:gromovgeneral1} is equivalent
to problem \eqref{eq:gromovcov}. 
\end{proof}
The left-hand term of \eqref{eq:gromovcov} is closely related to the sum of symmetric co-kurtosis and so depends on co-moments of order $ 4$ 
of $ \op $. On the other hand, the right-hand term is directly related to the co-moments of order $2$ of $\op$. For this reason, problem \eqref{eq:gromovcov}
is hard to solve because it involves to optimize simultaneously the co-moments of order $2$ and $4$ of $\op$ and so to know the probabilistic rule which 
links them. This rule is well-known when $ \op $ is Gaussian (Isserlis lemma) but this is not the case in general to the best of our knowledge and
there is no reason for the solution of problem \eqref{eq:gromovcov} to be Gaussian.

\section{Study of the general problem}\label{sec:3}
Since problem \eqref{eq:gromovcov} is hard to solve because of its dependence on co-moments of order $2 $ and $ 4$ of $ \op $, one can optimize 
both terms seperately in order to find a lower bound of $ GW_2(\m,\n)$. In the rest of the paper we suppose for convenience and without any
loss of generality that $ n \leq m $. 

\begin{theorem}\label{thm:lowbound}
Suppose without any loss of generality that $ n \leq m $. Let $ \m = \mathcal{N}(m_0,\cov_0) $ and
$ \n = \mathcal{N}(m_1,\cov_1) $ be two Gaussian measures on $ \reel^m $ and $ \reel^n $. Let $ P_0,D_0 $ and $ P_1,D_1 $ be the respective diagona\-lizations of  $\cov_0 (= P_0D_0P_0^T)$ and $ \cov_1 (=P_1D_1P_1^T) $ which sort eigenvalues in decreasing order. We suppose that $ \cov_0 $ is non-singular. A lower bound for $ GW_2(\m,\n) $ is then 
\begin{equation}\label{eq:lowerboundgromov}
GW^2_2(\m,\n) \geq  LGW_2^2(\m,\n), 
\end{equation}
where
\begin{equation}\label{eq:LGW}\tag{LGW}
\begin{split}
LGW_2^2(\m,\n) &=  4\left(\textup{tr}(D_0) - \textup{tr}(D_1)\right)^2 + 4\left(\|D_0\|_{\spa{F}} -\|D_1\|_{\spa{F}}\right)^2 + 4\|D_0^{(n)}-D_1\|_{\spa{F}}^2 \\ &\quad + 4\left(\|D_0\|^2_{\spa{F}} - \|D_0^{(n)}\|_{\spa{F}}^2\right).
\end{split}
\end{equation}
\end{theorem} 
The proof of this theorem is divided in smaller intermediary results. First we recall the Isserlis lemma (see \cite{Isserlis}), which allows to derive 
the co-moments of order $ 4 $ of a Gaussian distribution as a fonction of its co-moments of order $2$.

\begin{restatable}[\textbf{Isserlis, 1918, \cite{Isserlis}}]{lemma}{Isserlis}\label{lemme:Isserlis}
Let X be a zero-mean Gaussian vector of size $n $. Then 
\begin{equation}\label{eq:Isserlis}
\forall i,j,k,l \leq n, \  \expe[X_iX_jX_kX_l] = \expe[X_iX_j]\expe[X_kX_l] + \expe[X_iX_k]\expe[X_jX_l] + \expe[X_iX_l]\expe[X_jX_k].
\end{equation}
\end{restatable} 

Then we derive the following general optimization lemmas.  The proofs of these two lemmas are postponed to the Appendix (Section \ref{sec:appendix}).

\begin{restatable}{lemma}{maxnorm}\label{lemme:maxnorm2}
    Suppose that $ n \leq m $. Let $ \Sigma $ be a semi-definite positive matrix of size $ m + n $ of the form 
    $$
    \cov = \begin{pmatrix} \cov_0 &  K \\ K^T & \cov_1 \end{pmatrix}, 
    $$
    with $ \cov_0 \in S^{++}_m(\mathbb{R}) $, $ \cov_1 \in S^{+}_n(\mathbb{R}) $ and $ K  \in \reel^{m \times n}$. 
    Let $ P_0,D_0 $ and $ P_1, D_1 $ be the respective diagonalisations of $ \cov_0 (= P_0^TD_0P_0) $ and $ \cov_1 (= P_1^TD_1P_1)$ which sort 
    the eigenvalues in decreasing order. Then
    \begin{equation}\label{eq:maxnorm2}
    \max_{\cov_1 - K^T\cov_0^{-1}K \in S^+_n(\mathbb{R})} \|K\|^2_{\mathcal{F}}  = \textup{tr}(D_0^{(n)}D_1),
    \end{equation}
    and is achieved at any
    \begin{equation}\label{eq:optimalKgaussian}\tag{opKl2}
    K^* = P_0^T\begin{pmatrix} \tilde{I}_n(D_0^{(n)})^\frac{1}{2}D_1^\frac{1}{2} \\ 0_{m-n,n} \end{pmatrix}P_1,
    \end{equation}
    where $ \tilde{I}_n $ is of the form $ \textup{diag}((\pm 1)_{i\leq n}) $.
\end{restatable}

\begin{restatable}{lemma}{maxnormm}\label{lemme:maxnorm1}
    Suppose that  $ n \leq m $. Let $\Sigma $ be a semi-definite positive matrix of size  $ m + n $ of the form:
    $$  \Sigma = \begin{pmatrix} \cov_0 & K \\ K^T & \cov_1\end{pmatrix}, $$
    where $ \cov_0 \in S^{++}_m(\mathbb{R}) $, $ \cov_1 \in S^{+}_n(\mathbb{R}) $, and $ K  \in \reel^{m \times n}$. Let $ A \in \reel^{n \times m} $
    be a matrix with rank $1$. Then
    \begin{equation}
    \max_{\cov_1 - K^T\cov_0^{-1}K \in S_n^+(\mathbb{R})} \textup{tr}(KA) = \sqrt{\textup{tr}(A\cov_0A^T\cov_1)}.
    \end{equation}
    In particular, if $ \cov_0 = \textup{diag}(\alpha) $ and $ \cov_1 = \textup{diag}(\beta) $ with $ \alpha \in \reel^m $ and $ \beta \in \reel^n $, then
    \begin{equation}\label{eq:maxnorm1}
    \max_{\cov_1 - K^T\cov_0^{-1}K \in S_n^+(\mathbb{R})}  \textup{tr}(K\mathbb{1}_{n,m}) = \sqrt{\textup{tr}(\cov_0)\textup{tr}(\cov_1)},
    \end{equation}
    with $ \mathbb{1}_{n,m} = (1)_{i \leq n, j \leq m} $, and is achieved at 
    \begin{equation}\label{eq:optimalKnorm1}\tag{opKl1}
    K^* = \frac{\alpha\beta^T}{\sqrt{\textup{tr}(\cov_0)\textup{tr}(\cov_1)}}.
    \end{equation}
\end{restatable}

\begin{proof}[Proof of theorem \ref{thm:lowbound}] 
For $ \m = \mathcal{N}(m_0,\cov_0) $ and $ \n = \mathcal{N}(m_1,\cov_1) $, we note $ P_0,D_0 $
and $ P_1,D_1 $ the respective diagonalizations of $ \cov_0 $ and $ \cov_1 $ which sort the eigenvalues in decreasing order. 
Let $ T_0 : x \in \reel^m \mapsto P_0^T(x- m_0) $ and $ T_1 :  y \in \reel^n \mapsto P_1^T(y-m_1) $. For $ \op \in \Pi(T_0\#\m,T_1\#\n) $ and $ (X,Y) \sim \op $,  we denote $ \cov $ the covariance matrix 
of $ \op $ and $ \tilde{\cov} $ the covariance matrix of $ (X^2,Y^2)$ with $ X^2 := ([XX^T]_{i,i})_{i \leq m} $ and $ Y^2 := ([YY^T]_{j,j})_{j \leq n} $. Using Isserlis lemma to compute $ \text{Cov}(X^2,X^2) $ and $ \text{Cov}(Y^2,Y^2) $, it comes that $ \cov $ and $ \tilde{\cov} $ are of the form:
\begin{equation}
\cov = \begin{pmatrix} D_0 & K \\ K^T & D_1 \end{pmatrix} \quad \text{and} \quad \tilde{\cov} = \begin{pmatrix} 2D_0^2 & \tilde{K} \\ \tilde{K}^T & 2D_1^2 \end{pmatrix}.
\end{equation}
In order to find a supremum for each term of \eqref{eq:gromovcov}, we use a necessary condition for $ \op $ to be in $ \Pi(T_0\#\m,T_1\#\n) $
which is that $ \cov $ and $ \tilde{\cov} $ must be semi-definite positive. To do so, we can use the equivalent condition that the Schur complements of $ \cov $ 
and $ \tilde{\cov} $, namely $ D_1 - K^TD_0^{-1}K $ and $ 2D_1^2 - \frac{1}{2}\tilde{K}^TD_0^{-2}\tilde{K} $,  must also be semi-definite positive. Remarking that the left-hand term in \eqref{eq:gromovcov} can be rewritten $ \text{tr}(\tilde{K}\mathbb{1}_{m,n}) $, 
we have the two following inequalities

\begin{equation}\label{eq:inequalityorder4}
\sup_{X \sim T_0\#\m, Y \sim T_1\#\n}\sum\limits_{i,j}\text{Cov}(X_i^2,Y_j^2)   \leq \max_{2D_1^2 - \frac{1}{2}K^TD_0^{-2}K \in S^{+}_n(\mathbb{R})} \text{tr}(\tilde{K}\mathbb{1}_{n,m}),
\end{equation}
and
\begin{equation}\label{eq:inequalityorder2}
\sup_{X \sim T_0\#\m, Y \sim T_1\#\n} \|\text{Cov}(X,Y)\|^2_{\spa{F}} \leq \max_{D_1 - K^TD_0^{-1}K \in S^{+}_n(\mathbb{R})} \|K\|_{\spa{F}}^2. 
\end{equation}


\noindent Applying Lemmas \ref{lemme:maxnorm2} and  \ref{lemme:maxnorm1} on both right-hand terms, 
we get on one hand:
\begin{equation}\label{eq:solproblemnorm2}
\sup_{X \sim T_0\#\m, Y \sim T_1\#\n} \|\text{Cov}(X,Y)\|^2_{\spa{F}} \leq \text{tr}(D_0^{(n)}D_1), 
\end{equation}  
and on the other hand:
\begin{equation}\label{eq:solproblemnorm1}
\begin{split}
\sup_{X \sim T_0\#\m, Y \sim T_1\#\n} \sum\limits_{i,j}\text{Cov}(X^2_i,Y^2_j) &\leq 2\sqrt{\text{tr}(D_0^2)\text{tr}(D_1^2)} \\
 &\quad = 2\|D_0\|_{\spa{F}}\|D_1\|_{\spa{F}}.
\end{split}
\end{equation} 
Furthermore, using Lemma \ref{lemme:vayer}, it comes that
\begin{equation}
\begin{split}
GW_2^2(\m,\n) &= C_{\m,\n} \\
&\quad - 4\sup_{X \sim T_0\#\m, Y \sim T_1\#\n}\left(\sum\limits_{i,j}\text{Cov}(X_i^2,Y_j^2) + \sum\limits_{i,j}\expe[X_i^2]\expe[X_j^2] + 2\left\|\text{Cov}(X,Y)\right\|_{\spa{F}}^2\right) \\
&\geq C_{\m,\n} - 8\sqrt{\text{tr}(D_0^2)\text{tr}(D_1^2)} - 4\text{tr}(D_0)\text{tr}(D_1) - 8\text{tr}(D_0^{(n)}D_1),
\end{split}
\end{equation}
where
\begin{equation}
\begin{split}
C_{\m,\n} &= \expe_{U \sim \mathcal{N}(0,2D_0)}[\|U\|^4_{\reel^m}] + \expe_{V \sim \mathcal{N}(0,2D_1)}[\|V\|^4_{\reel^n}] - 4\expe_{X \sim \m}[\|X\|^2_{\reel^m}]\expe_{Y \sim \n}[\|Y\|^2_{\reel^n}] \\
&=  8\text{tr}(D_0^2) + 4(\text{tr}(D_0))^2 + 8\text{tr}(D_1^2) + 4(\text{tr}(D_1))^2 - 4\text{tr}(D_0)\text{tr}(D_1).
\end{split}
\end{equation}
Finally
\begin{equation}
\begin{split}
GW_2^2(\m,\n) &\geq 4(\text{tr}(D_0))^2 + 4(\text{tr}(D_1))^2 - 8\text{tr}(D_0)\text{tr}(D_1) + 8\text{tr}(D_0^2) + 8\text{tr}(D_1^2) \\
& \quad- 8\sqrt{\text{tr}(D_0^2)\text{tr}(D_1^2)} - 8\text{tr}(D_0^{(n)}D_1) \\
& \quad = LGW_2^2(\m,\n).
\end{split}
\end{equation}
\end{proof}

Inequalities \eqref{eq:inequalityorder4} and \eqref{eq:inequalityorder2} become equalities if one can exhibit a plan $ \op $ such that $ \cov $ or $ \tilde{\cov} $ are
such that $ \|K\|_{\spa{F}}^2 $ or $ \text{tr}(\tilde{K}\mathbb{1}_{n,m}) $ are maximized.
This is the case for \eqref{eq:inequalityorder2} where we can exhibit the Gaussian
plan $ \op^* $ such that $ K $ is of the form \eqref{eq:optimalKgaussian} 
but it seems however more tricky to exhibit such a plan for inequality \eqref{eq:inequalityorder4}. Indeed, it can be shown 
that it doesn't exist a Gaussian plan such that $ \tilde{K} $ is of the form \eqref{eq:optimalKnorm1}.

The lower bound $ LGW_2 $ is reached if it exists a plan $ \op $ which optimizes both terms simultaneously. This seems rather unlikely because
if a probability distribution has its covariance matrix such that $ K $ is of the form \eqref{eq:optimalKgaussian}, then it is necessarily Gaussian
thanks to the equality case in Cauchy-Schwarz: if $ D_0 = \text{diag}(\alpha) $ and $ D_1 = \text{diag}(\beta) $ with $ \alpha \in \reel^m $ and 
$ \beta \in \reel^n $, and if $ \op $ has its covariance matrix such that $ K$ is of the form \eqref{eq:optimalKgaussian}, then for all $ i \leq n $,
$ \text{Cov}(X_i,Y_i) = \pm \sqrt{\alpha_i\beta_i} $ and $ Y_i $ depends linearly in $ X_i $ . As an outcome, $ \op $ is Gaussian and we can compute, using
Isserlis lemma, that $ \text{tr}(\tilde{K}\mathbb{1}_{n,m}) = 2\text{tr}(D_0D_1) $ and so $ \tilde{K}$ cannot be of
the form \eqref{eq:optimalKnorm1}. However, we didn't prove that the solution of the form \eqref{eq:optimalKgaussian} is unique so it may exist
another solution which doesn't imply that $ \op $ has to be Gaussian.

\section{Problem restricted to Gaussian transport plans}\label{sec:4}
In this section, we study the following problem, where we constrain the optimal transport plan to be Gaussian. 
\begin{equation}\label{eq:gromovrestrictedgaussian}\tag{GaussGW}
GGW_2^2(\m,\n) = \inf_{\op \in \Pi(\m,\n)\cap\mathcal{N}_{m+n}} \int \int \left(\|x - x'\|^2_{\reel^m} - \|y - y'\|^2_{\reel^n}\right)^2d\op(x,y)d\op(x',y'),
\end{equation}
where  $ \spa{N}_{m + n} $ is the set of Gaussian measures on $ \reel^{m + n} $. We show the following main result.

\begin{theorem}\label{thm:gaussian}
Suppose without any loss of generality that $ n \leq m $. Let $ \m = \mathcal{N}(m_0,\cov_0) $ and $ \n = \mathcal{N}(m_1,\cov_1) $ be two Gaussian measures on $ \reel^m $ and $ \reel^n $. Let $ P_0,D_0 $ and $ P_1,D_1 $ be the respective diagona\-lizations of  $\cov_0 (= P_0D_0P_0^T)$ and $ \cov_1 (=P_1D_1P_1^T) $ which sort eigenvalues in decreasing order. We suppose that $ \cov_0 $ is non-singular ($ \m $ is not degenerate).
Then problem \eqref{eq:gromovrestrictedgaussian} admits a solution of the form $ \op^* = (I_m, T)\#\m $ with $ T $ affine of the form
\begin{equation}\label{eq:optimalmapgromov} 
\forall x \in \reel^m, \   T(x) = m_1 + P_1AP_0^T(x - m_0).
\end{equation}
where $ A \in \mathbb{R}^{n \times m} $ is written
$$ A = \begin{pmatrix} \tilde{I}_nD_1^{\frac{1}{2}}(D_0^{(n)})^{-\frac{1}{2}} & 0_{n,m-n} \end{pmatrix}, $$
where $ \tilde{I}_n $ is of the form $ \textup{diag}((\pm 1)_{i\leq n}) $. Moreover
\begin{equation}\label{eq:GGW}\tag{GGW}
GGW^2_2(\m,\n) = 4(\textup{tr}(D_0) - \textup{tr}(D_1))^2 + 8\|D_0^{(n)} - D_1\|_{\spa{F}}^2 + 8\left(\|D_0\|^2_{\spa{F}} - \|D_0^{(n)}\|_{\spa{F}}^2\right).
\end{equation}
\end{theorem}

\begin{proof} This theorem is a direct consequence of Isserlis lemma \ref{lemme:Isserlis}: indeed, the left term in equation
\eqref{eq:gromovcov} can be in that case rewritten $  2\|\text{Cov}(X,Y)\|_{\spa{F}}^2 $ and so problem \eqref{eq:gromovrestrictedgaussian} is 
equivalent to
\begin{equation}
\sup_{X \sim T_0\#\m, Y \sim T_1\#\n} \|\text{Cov}(X,Y)\|_{\spa{F}}^2. 
\end{equation}
Applying Lemma \ref{lemme:maxnorm2}, we can exhibit a Gaussian optimal plan $ \op^* \in \Pi(\m,\n ) $ with covariance matrix $ \cov $ of the form:
\begin{equation}
\cov = \begin{pmatrix} \cov_0 & K^* \\ K^{*T} & \cov_1 \end{pmatrix},
\end{equation}
with     
\begin{equation}
K^* = P_0^T\begin{pmatrix} \tilde{I}_n(D_0^{(n)})^\frac{1}{2}D_1^\frac{1}{2} \\ 0_{m-n,n} \end{pmatrix}P_1.
\end{equation}
Thus, using the equality case in Cauchy-Schwarz, we can exhibit an optimal transport map $ T $ of the form 
\begin{equation}\label{eq:optimalT}
\forall x \in \reel^m, \ T(x) = m_1 + P_1AP_0^T(x - m_0),
\end{equation}
with 
$$ A = \begin{pmatrix} \tilde{I}_nD_1^{\frac{1}{2}}(D_0^{(n)})^{-\frac{1}{2}} & 0_{n,m-n} \end{pmatrix}, $$
where $ \tilde{I}_n $ is of the form $ \textup{diag}((\pm 1)_{i\leq n}) $. Moreover, using Lemmas \ref{lemme:vayer} and \ref{lemme:maxnorm2}, it comes that
\begin{equation}
\begin{split}
GGW_2^2(\m,\n) &= C_{\m,\n} - 16\sup_{X \sim T_0\#\m, Y \sim T_1\#\n} \|\text{Cov}(X,Y)\|_{\spa{F}}^2 \\
&= 8\text{tr}(D_0^2) + 4(\text{tr}(D_0))^2 + 8\text{tr}(D_1^2) + 4(\text{tr}(D_1))^2 - 4\text{tr}(D_0)\text{tr}(D_1) - 16\text{tr}(D_0^{(n)}D_1) \\
&= 4(\textup{tr}(D_0) - \textup{tr}(D_1))^2 + 8\textup{tr}\left((D_0^{(n)} - D_1)^2\right) + 8\left(\textup{tr}(D_0^2) - \text{tr}((D_0^{(n)})^2)\right).
\end{split}
\end{equation}
\end{proof}

\paragraph{Link with Gromov-Monge}
The previous result generalizes Theorem 4.2.6 in~\cite{These}, which studies the solutions of the linear Gromov-Monge problem between Gaussian distributions
\begin{equation}
 \inf_{T \# \m=\n,\, \text{T linear}} \int \int \left(\|x - x'\|^2_{\reel^m} - \|T(x) - T(x')\|^2_{\reel^n}\right)^2d\m(x)d\n(x').
\label{eq:linear_gromov-Monge}
\end{equation}
Indeed, solutions of~\eqref{eq:linear_gromov-Monge} necessarily provide Gaussian transport plans $\pi = (I_m,T)\#\m$  if $T$ is linear. Conversely, Theorem~\ref{thm:gaussian} shows that restricting the optimal plan to be Gaussian in Gromov-Wasserstein between two Gaussian distributions yields an optimal plan of the form $\pi = (I_m,T)\#\m$  with a linear $T$, whatever the dimensions $m$ and $n$ of the two Euclidean spaces.

\paragraph{Link with Principal Component Analysis}
We can easily draw connections between $ GGW_2^2 $ and PCA. 
Indeed, we can remark that the optimal plan can be derived by performing PCA on both distributions $ \m $ and $ \n $ in order to obtain distributions  $ \tilde{\m} $
and $ \tilde{\n} $ with zero mean vectors and diagonal covariance matrices with eigenvalues in decreasing order  ($\tilde{\m} = T_0\#\m$ and $\tilde{\n} = T_1\#\n$), then
by keeping only the $n$  first components in $ \tilde{\m} $ and finally by deriving the optimal transport plan which achieves $ W_2^2 $ between the obtained truncated distribution and
$ \tilde{\n} $. In other terms, noting $ P_n: \reel^m \rightarrow \reel^n $ the linear mapping which, for $ x \in \reel^m $ keeps only its $ n $ first components ($n$-frame), 
$ T_{W_2} $ the optimal transport map such that $ \op_{W_2} = (I_n,T_{W_2})\#P_n\#\tilde{\m} $ achieves $ W_2(P_n\#\tilde{\m},\tilde{\n}) $, it comes that the optimal plan $ \op_{GGW_2} $ which achieves $ 
GGW_2(\tilde{\m},\tilde{\n}) $ can be written
\begin{equation}
\op_{GGW_2} = (I_m,\tilde{I}_n\#T_{W_2}\#P_n)\#\tilde{\m}.
\end{equation}
An example of $ \op_{GGW_2} $ can be found in Figure \ref{fig:ggw2} when $ m = 2 $ and $ n = 1 $.

\begin{figure}[!h]
    \centering
    \includegraphics[width=\textwidth]{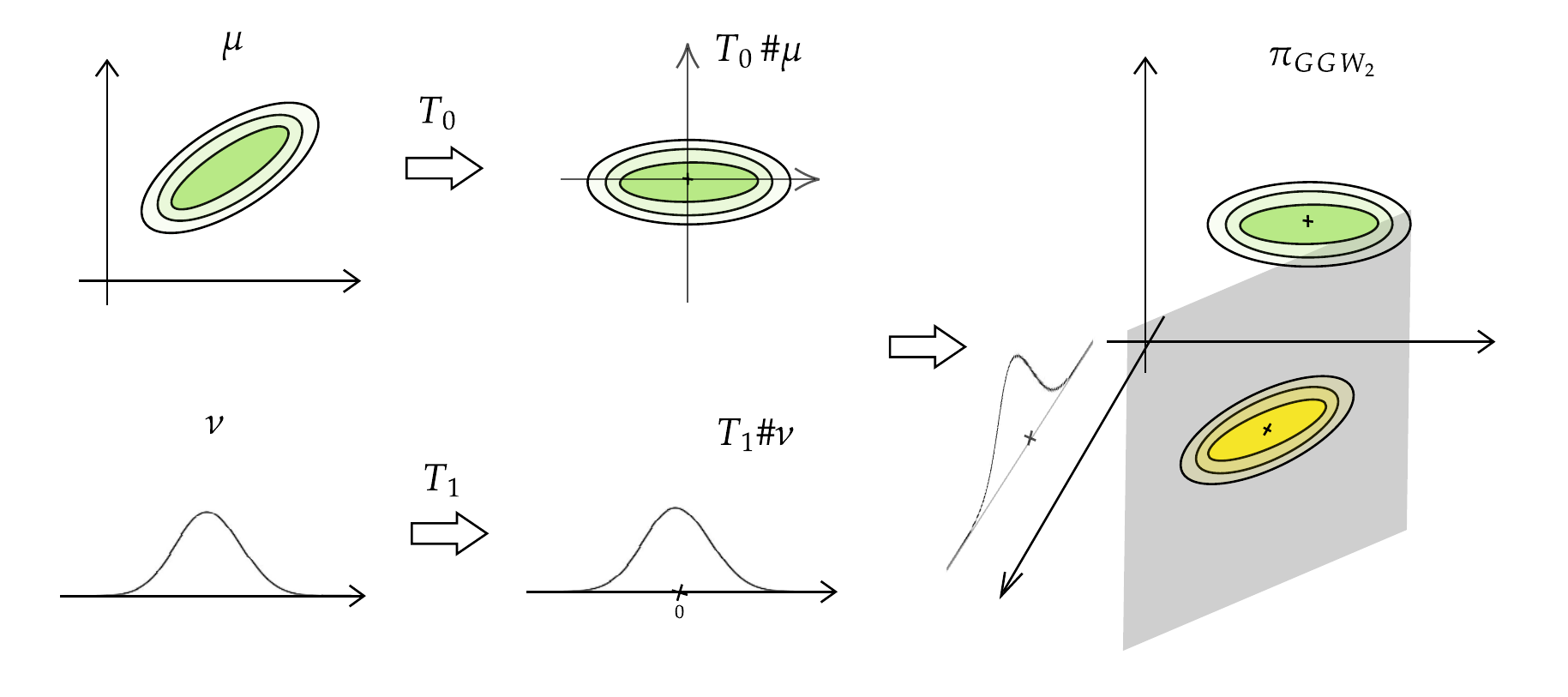}
    \caption{transport plan $ \op_{GGW_2} $ solution of problem \eqref{eq:gromovrestrictedgaussian} with $ m = 2 $  and $ n =1 $.
     In that case, $ \op_{GGW_2} $ is the degenerate Gaussian distribution supported by the plan of equation $ y = T_{W_2}(x) $, where $ T_{W_2} $ is the classic $ W_2 $ optimal transport map when the distributions
    are rotated and centered first. }
    \label{fig:ggw2}
    \centering
    \end{figure}

\paragraph{Case of equal dimensions}
When $ m = n $, the optimal plan $ \op_{GGW_2} $ which achieves $ GGW_2(\m,\n) $  is closely related to the optimal transport plan $ \op_{W_2} = (I_m,T_{W_2})\#T_0\#\m $. Indeed, $ \op_{GGW_2} $ can be simply derived by applying the transformations $ T_0 $ and $ T_1 $ to respectively $ \m $ and $ \n$, then by 
computing $ \op_{W_2} $ between $ T_0\#\m $  and $ T_1\#\n $, and finally by applying the inverse transformations $ T_0^{-1} $ and $ T_1^{-1} $. In other terms, $ \op_{GGW_2} $
can be written
\begin{equation}
\op_{GGW_2} = (I_m,T_1^{-1}\#\tilde{I}_n\#T_{W_2}\#T_0)\#\m.
\end{equation}
An example of transport between two Gaussians measures in dimension $ 2 $ in Figure \ref{fig:ggw2dim2}.

\begin{figure}[!h]
    \centering
    \includegraphics[width=\textwidth]{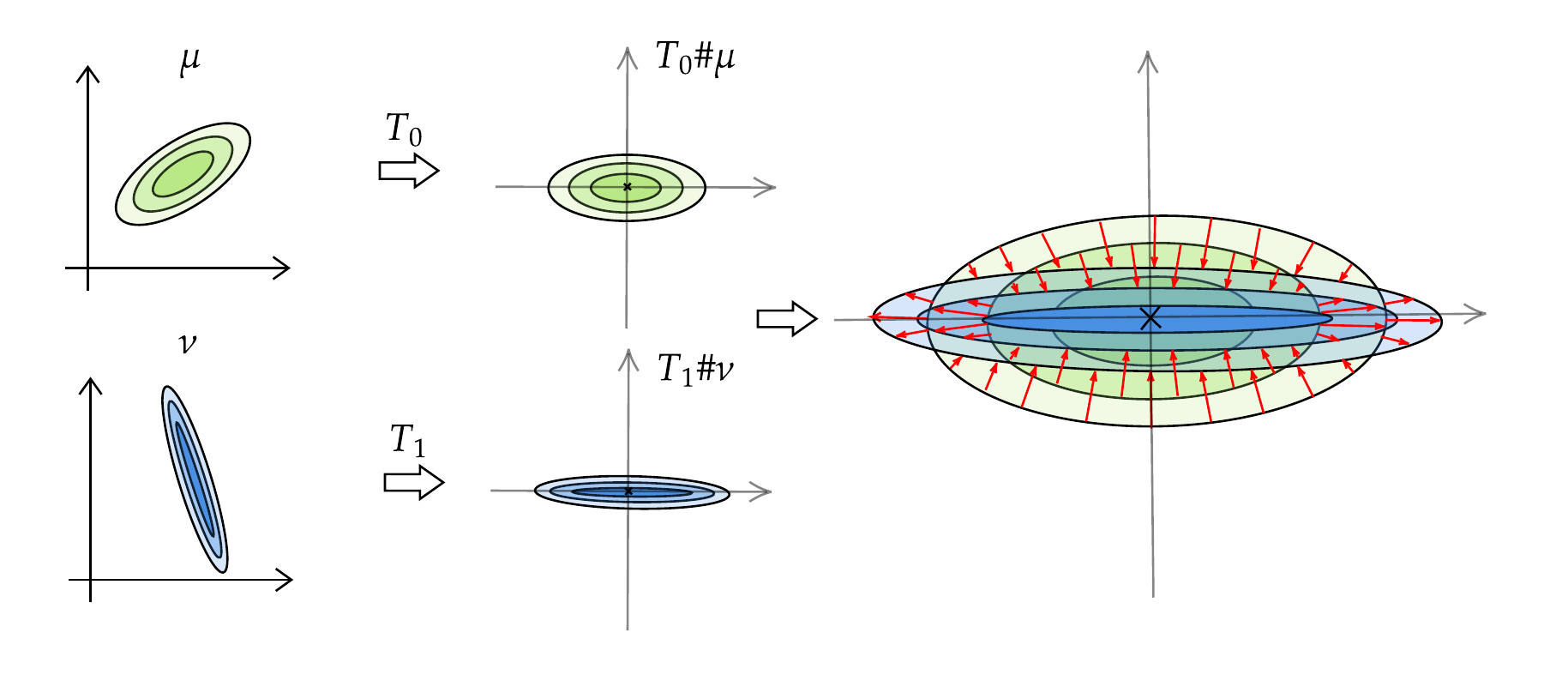}
    \caption{Solution of \eqref{eq:gromovrestrictedgaussian} between two Gaussians measures in dimension $ 2 $. First the distributions are centered and rotated. Then a classic $ W_2 $ transport is applied between the two aligned distributions. }
    \label{fig:ggw2dim2}
    \centering
    \end{figure}

As illustrated in Figure \ref{fig:W2vsGGW2}, the $ GGW_2 $ optimal transport map $ T_{GGW_2} $ defined in Equation \eqref{eq:optimalmapgromov} is not equivalent to the $ W_2$ optimal transport map $ T_{W_2} $ defined in \eqref{eq:mongemapgaussian} even when the dimensions $ m $ and $ n $ are equal. More precisely, it $ \cov_0 $ and $ \cov_1 $ can be diagonalized in the same orthonormal basis with eigenvalues in the same order (decreasing or increasing), then $ T_{W_2} $ and $ T_{GGW_2} $ are equivalent (top of Figure  \ref{fig:W2vsGGW2}). On the other hand, if $ \cov_0 $ and $ \cov_1 $ can be diagonalized in the same orthonormal basis but with eigenvalues not in the same order,  $ T_{W_2} $ and $ T_{GGW_2} $ will have very different behaviors (bottom of  Figure  \ref{fig:W2vsGGW2}). Between those two extreme cases, we can say that the closer the columns of $ P_0 $ will be collinear to the columns of $ P_1 $ (with the eigenvalues in decreasing order), the more $ T_{W_2} $ and $ T_{GGW_2} $ will tend to have similar behaviors (middle of Figure \ref{fig:W2vsGGW2}).

\begin{figure}[!h]
        \begin{subfigure}[b]{0.32\textwidth}
            \includegraphics[width=\textwidth]{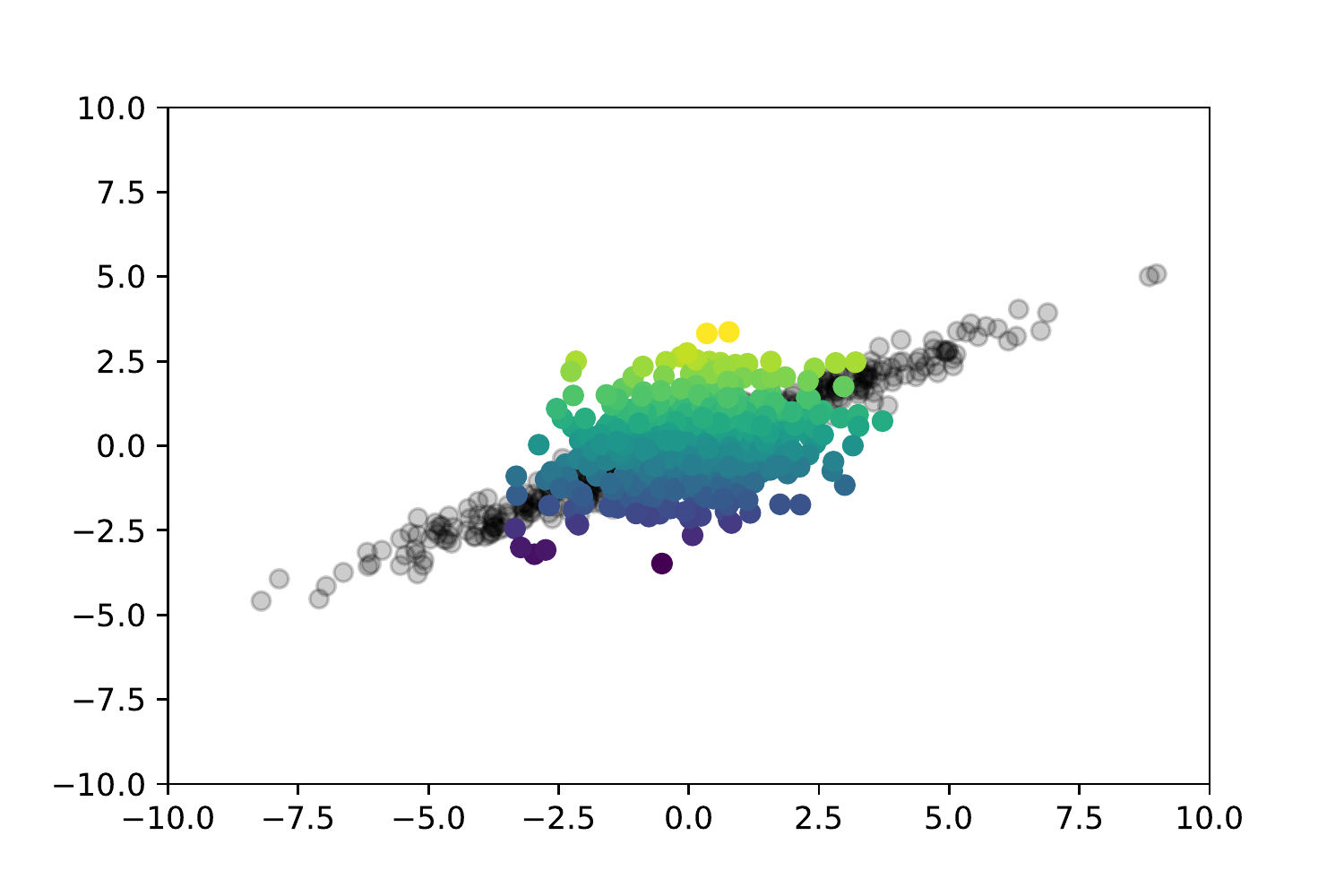}
        \end{subfigure}
        \begin{subfigure}[b]{0.32\textwidth}
            \includegraphics[width=\textwidth]{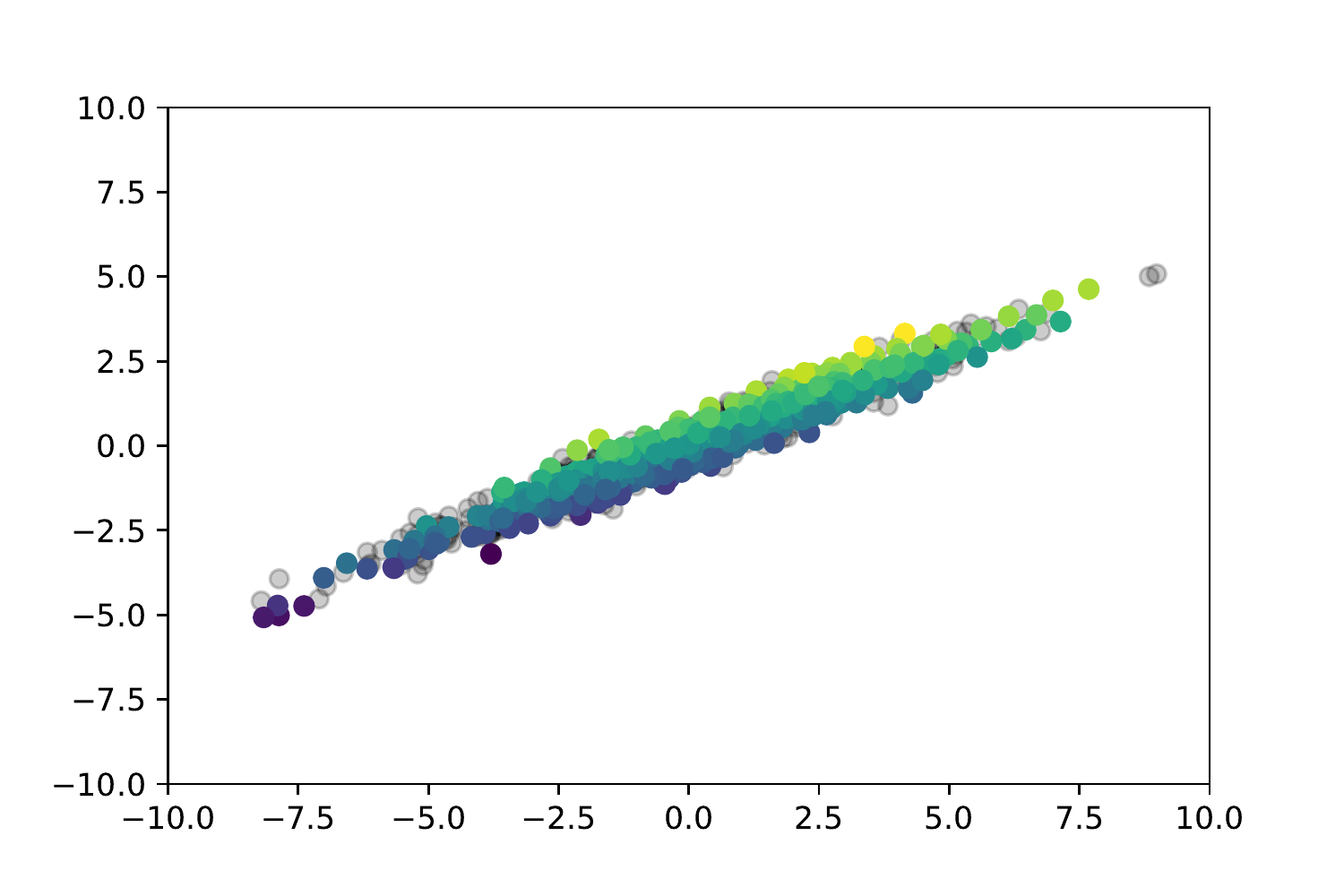}
        \end{subfigure}
        \begin{subfigure}[b]{0.32\textwidth}
            \includegraphics[width=\textwidth]{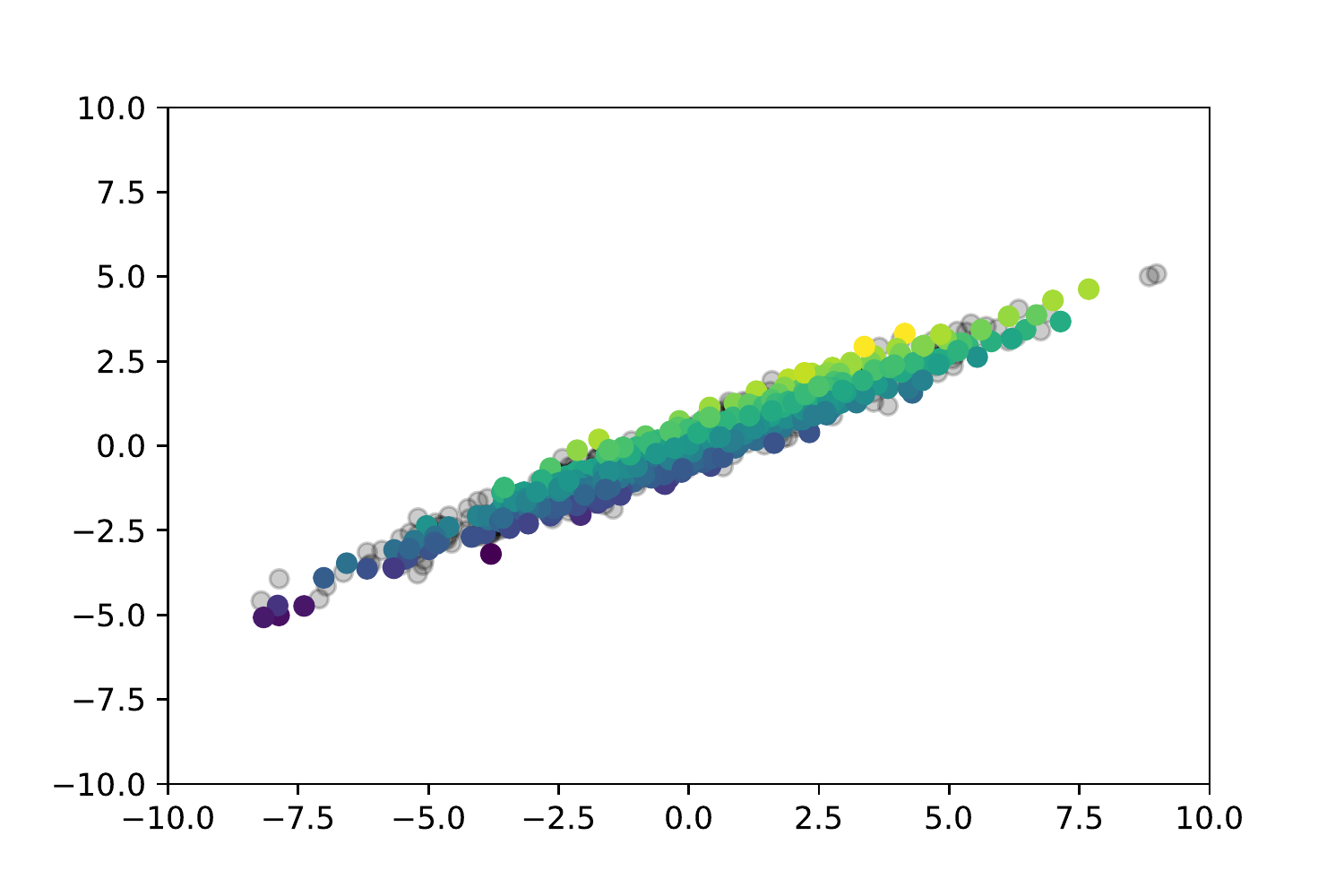}
        \end{subfigure}
        \centering
        \begin{subfigure}[b]{0.32\textwidth}
            \includegraphics[width=\textwidth]{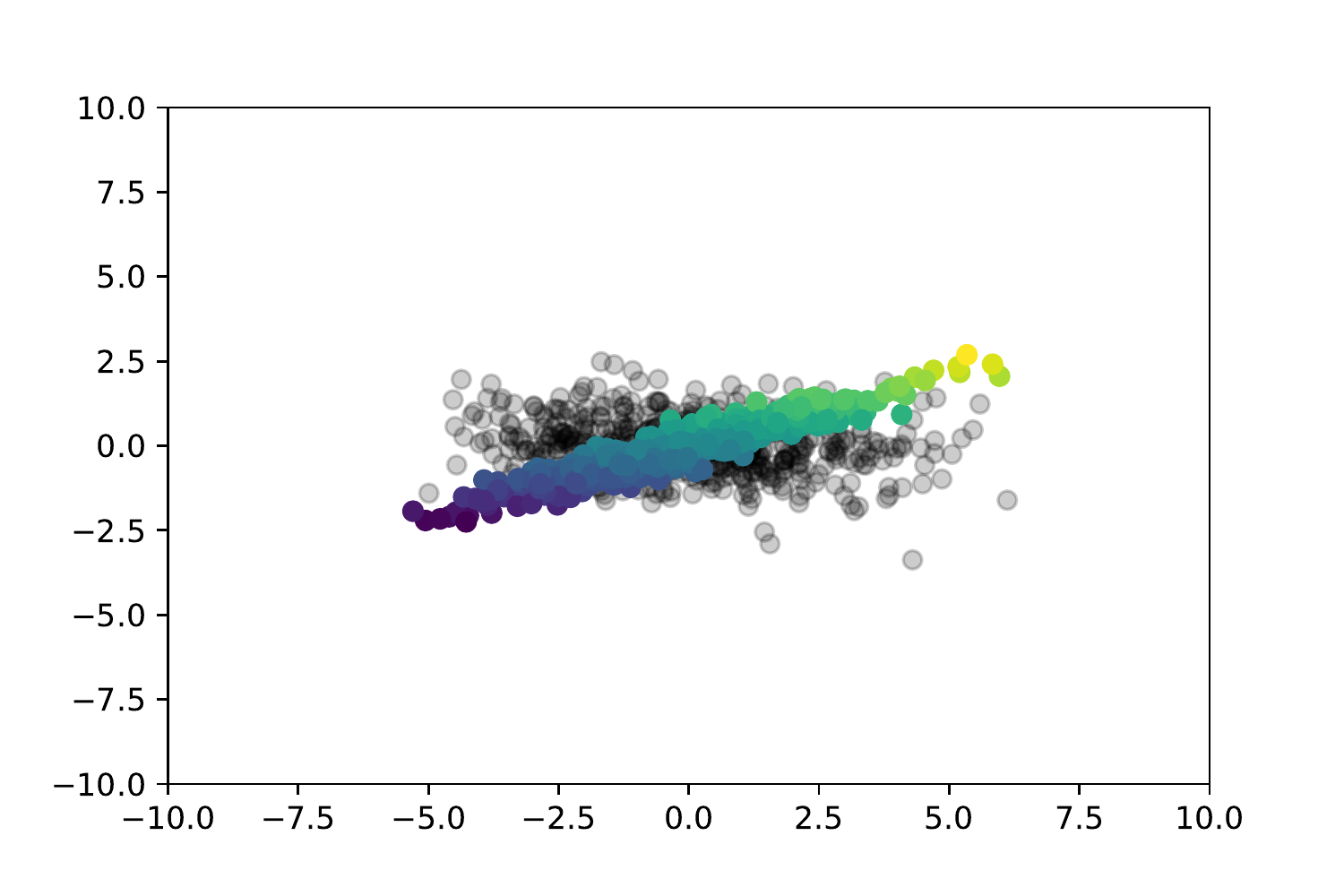}
        \end{subfigure}
        \begin{subfigure}[b]{0.32\textwidth}
            \includegraphics[width=\textwidth]{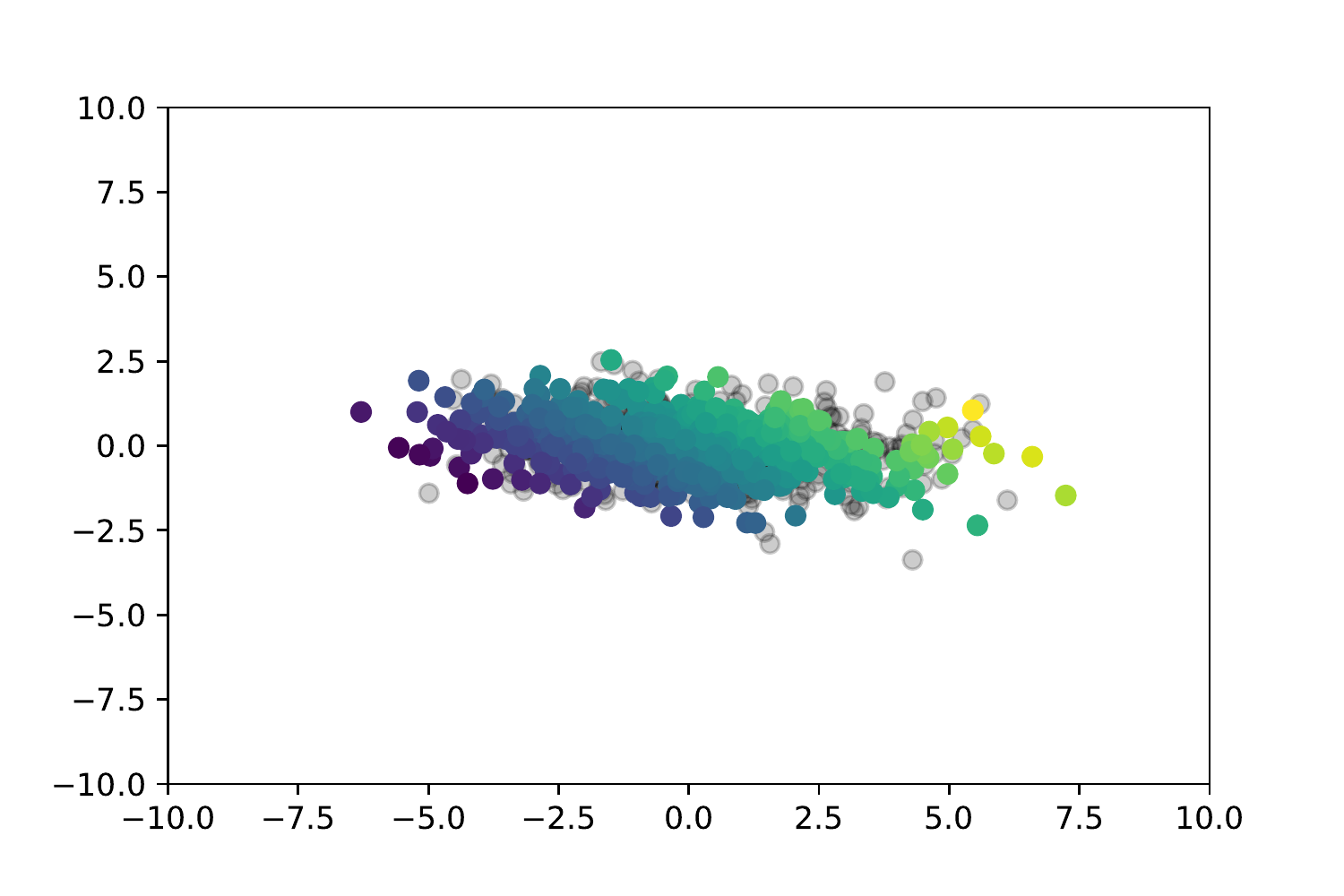}
        \end{subfigure}
        \begin{subfigure}[b]{0.32\textwidth}
            \includegraphics[width=\textwidth]{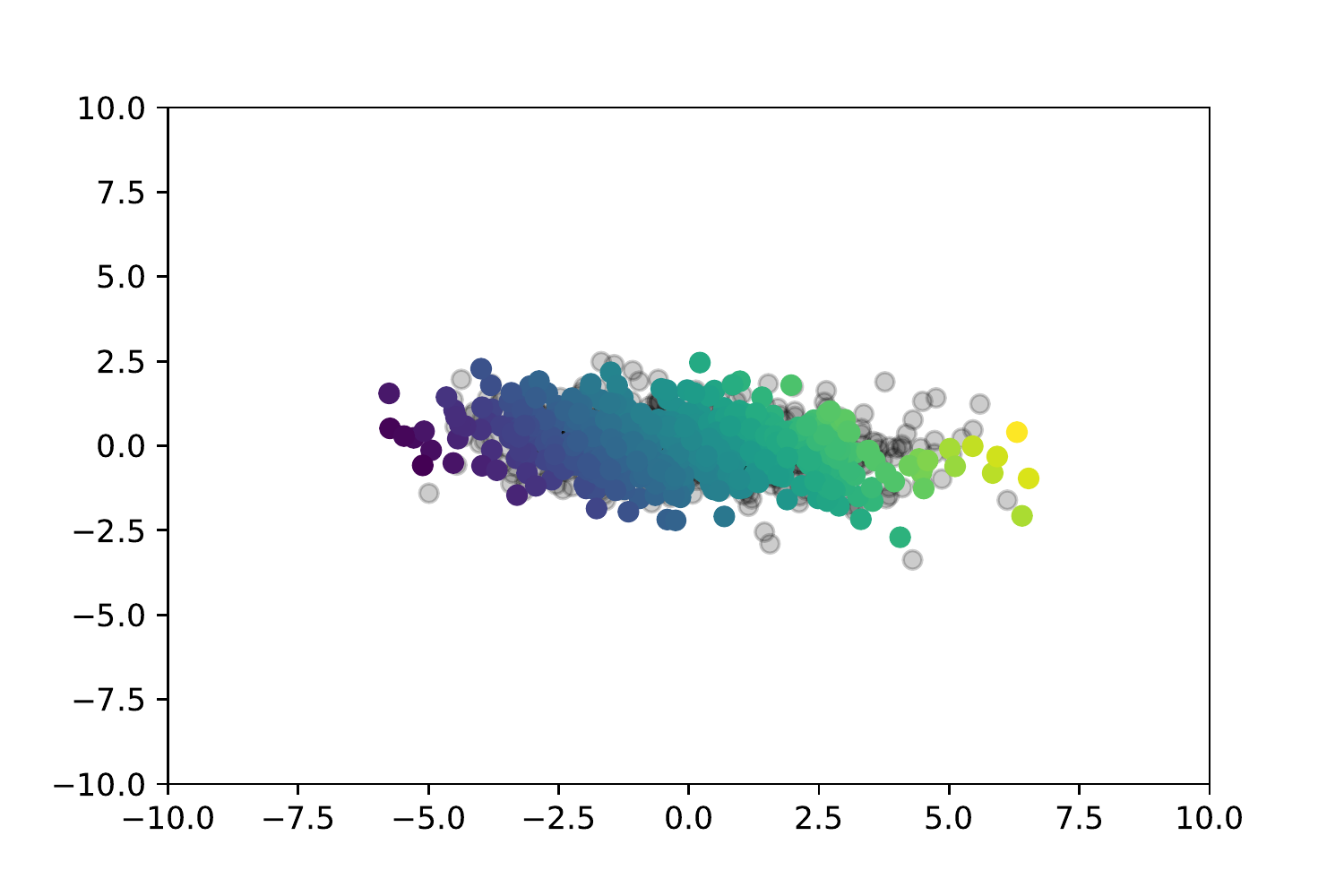}
        \end{subfigure}
        \begin{subfigure}[b]{0.32\textwidth}
            \includegraphics[width=\textwidth]{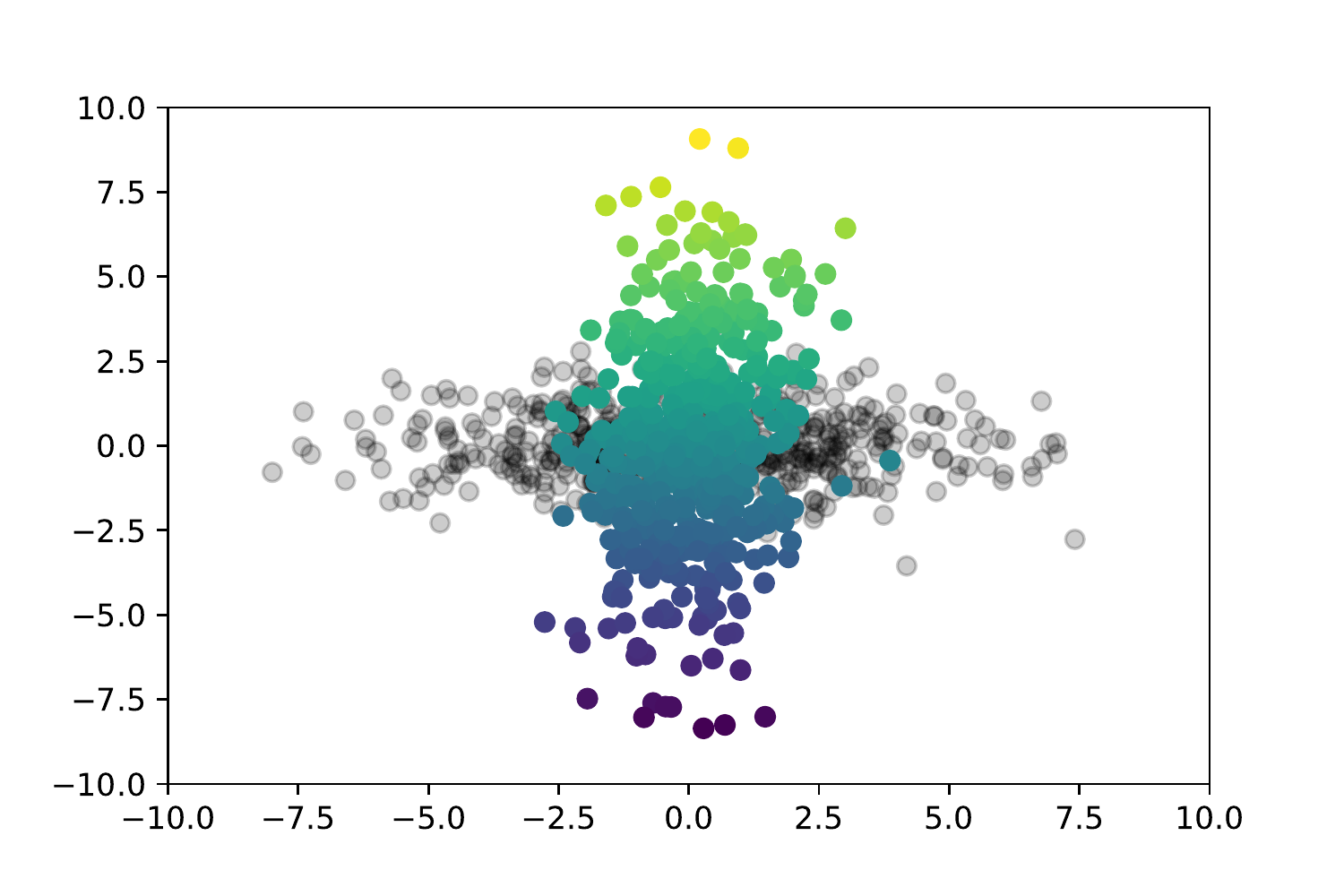}
        \end{subfigure}
        \begin{subfigure}[b]{0.32\textwidth}
            \includegraphics[width=\textwidth]{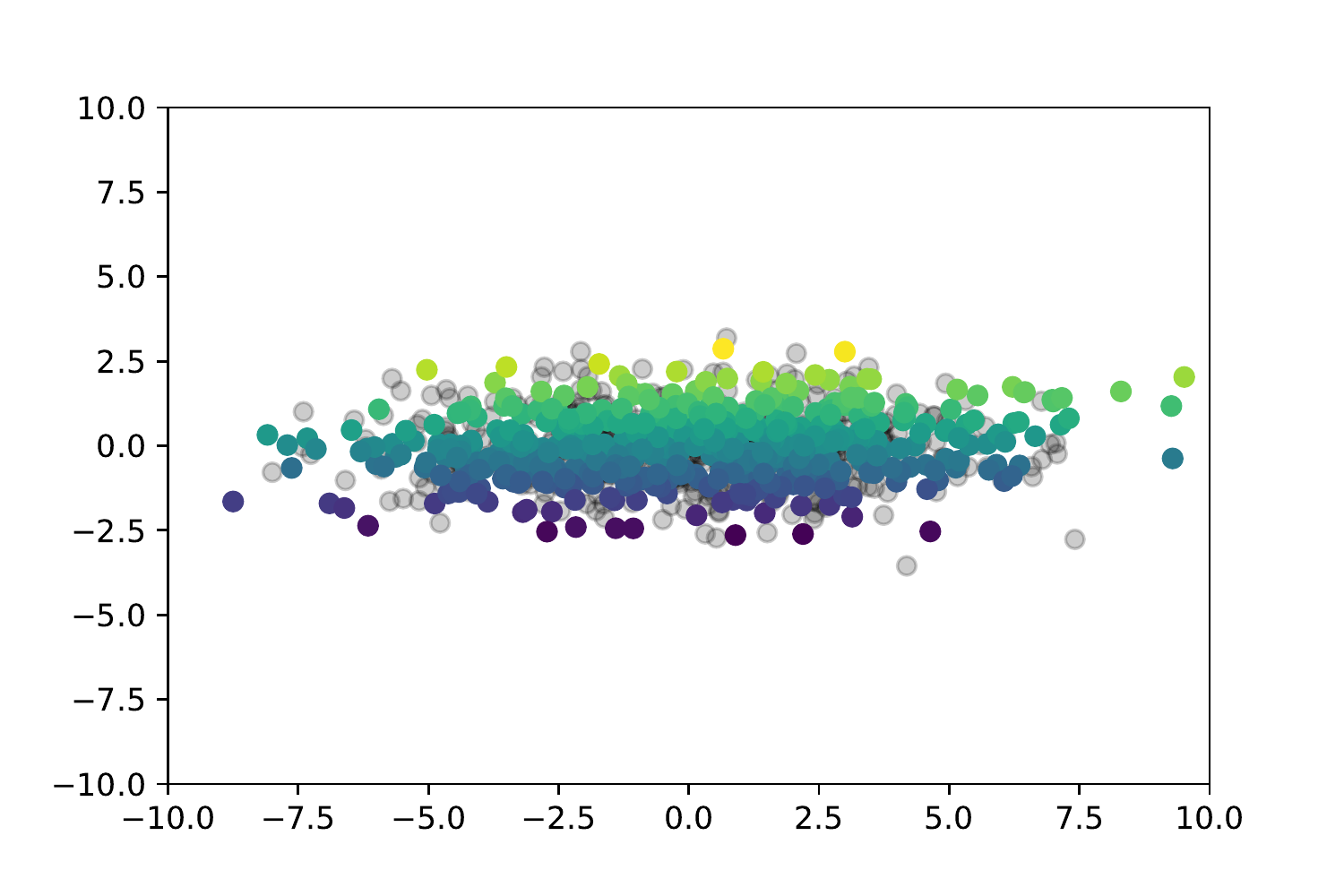}
        \end{subfigure}
        \begin{subfigure}[b]{0.32\textwidth}
            \includegraphics[width=\textwidth]{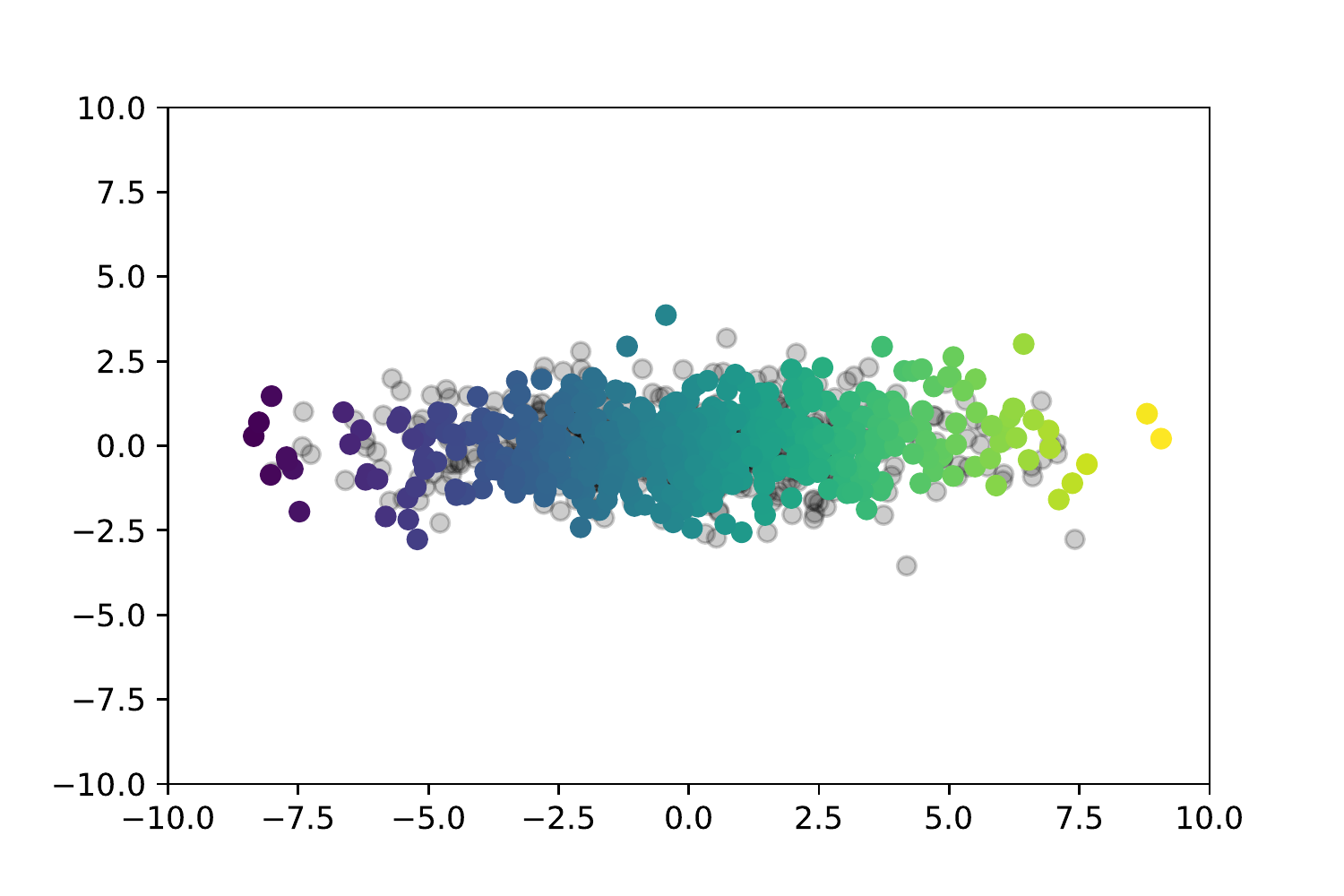}
        \end{subfigure}
    \caption{Comparison between $W_2 $  and $ GGW_2 $ mappings between empirical distributions. Left: 2D source distribution (colored) and target distribution (transparent). Middle: resulting mapping of Wasserstein $ T_{W_2} $. Right: resulting mapping of Gaussian Gromov-Wasserstein $ T_{GGW_2} $. The colors are added in order to visualize where each sample has been sent.}
    \label{fig:W2vsGGW2}
\end{figure}

\paragraph{Link with Gromov-Wasserstein with inner product as cost function}
If  $ \m $ and $ \n $ are centered Gaussian measures, let us consider the following problem
\begin{equation}\label{eq:innergw}\tag{innerGW}
GW_2^2(\langle .,. \rangle_m,\langle .,. \rangle_n,\m,\n) = \inf_{\op \in \Pi(\m,\n)} \int \int (\langle x,x'\rangle_m - \langle y,y'\rangle_n)^2d\op(x,y)d\op(x',y').
\end{equation}
Notice that the above problem is not restricted to Gaussian plans, but the following proposition shows that in fact its solution is Gaussian.
\begin{prop}
Suppose $ m \leq n $. Let $\m = \mathcal{N}(0,\cov_0) $ and $ \n =  \mathcal{N}(0,\cov_1) $ be two centered Gaussian measures respectively on $ \reel^m $ and $ \reel^n $. Then
the solution of problem \eqref{eq:gromovrestrictedgaussian} exhibited in theorem \ref{thm:gaussian} is also solution of problem \eqref{eq:innergw}.
\end{prop}
\begin{proof}  
The proof of this proposition is a direct consequence of lemma \ref{lemme:vayer}: indeed, applying it with $ a = 0 $, $ b = 0 $, and $ c = 1 $, it comes that problem \eqref{eq:innergw}
is equivalent to
\begin{equation}
\sup_{\op \in \Pi(\m,\n)} \left\|\int xy^Td\op(x,y)\right\|_{\spa{F}}^2.
\end{equation}
Since $ \m $  and $ \n $ are centered, it comes that problem \eqref{eq:innergw}
is equivalent to
\begin{equation}
\sup_{X \sim \m, Y \sim \n} \|\text{Cov}(X,Y)\|_{\spa{F}}^2.
\end{equation}
Applying Lemma \ref{lemme:maxnorm2}, it comes that the solution exhibited in Theorem \ref{thm:gaussian} is also solution of problem \eqref{eq:innergw}.
\end{proof}

Since $ GGW_2 $ is the Gromov-Wasserstein problem restricted to Gaussian transport plan, it is clear that \eqref{eq:gromovrestrictedgaussian} is an upper bound
of \eqref{eq:gromovgeneral1}. Combining this result with Theorem \ref{thm:lowbound}, we get the following simple but important result.

\begin{prop}\label{prop:frame} If $ \m = \mathcal{N}(m_0,\cov_0) $ and $ \n = \mathcal{N}(m_1,\cov_1) $ and $ \cov_0 $ is non-singular, then
\begin{equation}
LGW_2^2(\m,\n) \leq GW_2^2(\m,\n) \leq GGW_2^2(\m,\n).
\end{equation}
\end{prop}

\section{Tightness of the bounds and particular cases}\label{sec:5}

\subsection{Bound on the difference}

\begin{prop}\label{prop:bounds}
Suppose without loss of generality that $ n \leq m $, if $ \m = \mathcal{N}(m_0,\cov_0) $ and $ \n = \mathcal{N}(m_1,\cov_1)$, then
\begin{equation}
GGW_2^2(\m,\n) - LGW_2^2(\m,\n) \leq 8\|\cov_0\|_{\spa{F}}\|\cov_1\|_{\spa{F}}\left(1 - \frac{1}{\sqrt{m}}\right).
\end{equation}
\end{prop}

To prove this proposition, we will use the following technical result (the proof is postponed to the Appendix (Section \ref{sec:appendix})):
\begin{lemma}\label{lemme:innerproduct}
Let $ u \in \reel^m $ and $ v \in \reel^m  $ be two unit vectors with non-negative coordinates ordered in decreasing order. Then
\begin{equation}
u^Tv \geq \frac{1}{\sqrt{m}},
\end{equation}
with equality if $ u = (\frac{1}{\sqrt{m}},\frac{1}{\sqrt{m}},\dots)^T $ and $ v = (1,0,\dots)^T $.
\end{lemma}

\begin{proof}[Proof of Proposition \ref{prop:bounds}]
By subtracting \eqref{eq:LGW} from \eqref{eq:GGW}, it comes that
    \begin{equation}\label{eq:cauchy}
    \begin{split}
    GGW_2^2(\m,\n) - LGW_2^2(\m,\n) &= 8\left(\|D_0\|_{\spa{F}}\|D_1\|_{\spa{F}} - \text{tr}(D_0^{(n)}D_1)\right) \\
    &= 8\left(\|D_0\|_{\spa{F}}\|D_1^{[m]}\|_{\spa{F}} - \text{tr}(D_0D_1^{[m]})\right),
    \end{split}
    \end{equation}
    where $ D_1^{[m]} = \begin{pmatrix} D_1 & 0 \\ 0 & 0 \end{pmatrix} \in \reel^{m \times m} $. Noting $ \alpha \in \reel^m $ and $ \beta \in \reel^m $ the vectors
    of eigenvalues of $ D_0 $ and $ D_1^{[m]} $, it comes 
    \begin{equation}
    GGW_2^2(\m,\n) - LGW_2^2(\m,\n) = 8(\|\alpha\|\|\beta\| - \alpha^T\beta) = 8\|\alpha\|\|\beta\|(1 - u^Tv),
    \end{equation}
    where $ u = \frac{\alpha}{\|\alpha\|} $ and $ v = \frac{\beta}{\|\beta\|} $ .
    Applying lemma \ref{lemme:innerproduct}, we get directly that
    \begin{equation}
    \begin{split}
    GGW_2^2(\m,\n) - LGW_2^2(\m,\n) &\leq 8\|D_0\|_{\spa{F}}\|D_1^{[m]}\|_{\spa{F}}\left(1 - \frac{1}{\sqrt{m}}\right). \\
    &\quad = 8\|\cov_0\|_{\spa{F}}\|\cov_1\|_{\spa{F}}\left(1 - \frac{1}{\sqrt{m}}\right).
    \end{split}
    \end{equation}
    \end{proof}
    
    The difference between $ GGW_2^2(\m,\n) $ and $ LGW_2^2(\m,\n) $ can be seen as the difference between the right and left terms of the Cauchy-Schwarz inequality applied to the two vectors 
    of eigenvalues $ \alpha  \in \reel^m $ and $ \beta \in \reel^m $. The difference is maximized when the vectors $ \alpha $ and $ \beta $ are the least collinear possible. This happens when the eigenvalues of $ D_0 $ are all equal and $ n = 1 $ or $ \n $ is degenerate of true dimension $ 1 $. 
    On the other hand , this difference is null when $ \alpha $ and $ \beta $ are collinear. 
    Between those two extremal cases, we can say that the difference between $ GGW_2^2(\m,\n) $ and $ LGW_2^2(\m,\n) $ will be relatively small if the last $ m - n $ eigenvalues $ D_0 $ are small compared to the $ n $ 
    first eigenvalues and if the $ n $ first eigenvalues are close to be proportional to the eigenvalues of $ D_1 $. An example in the case where $ m = 2 $ and $ n = 1$ can be found in Figure \ref{fig:expe3}.

    \begin{figure}[!h]
        \centering
        \begin{subfigure}[b]{0.32\textwidth}
            \includegraphics[width=\textwidth]{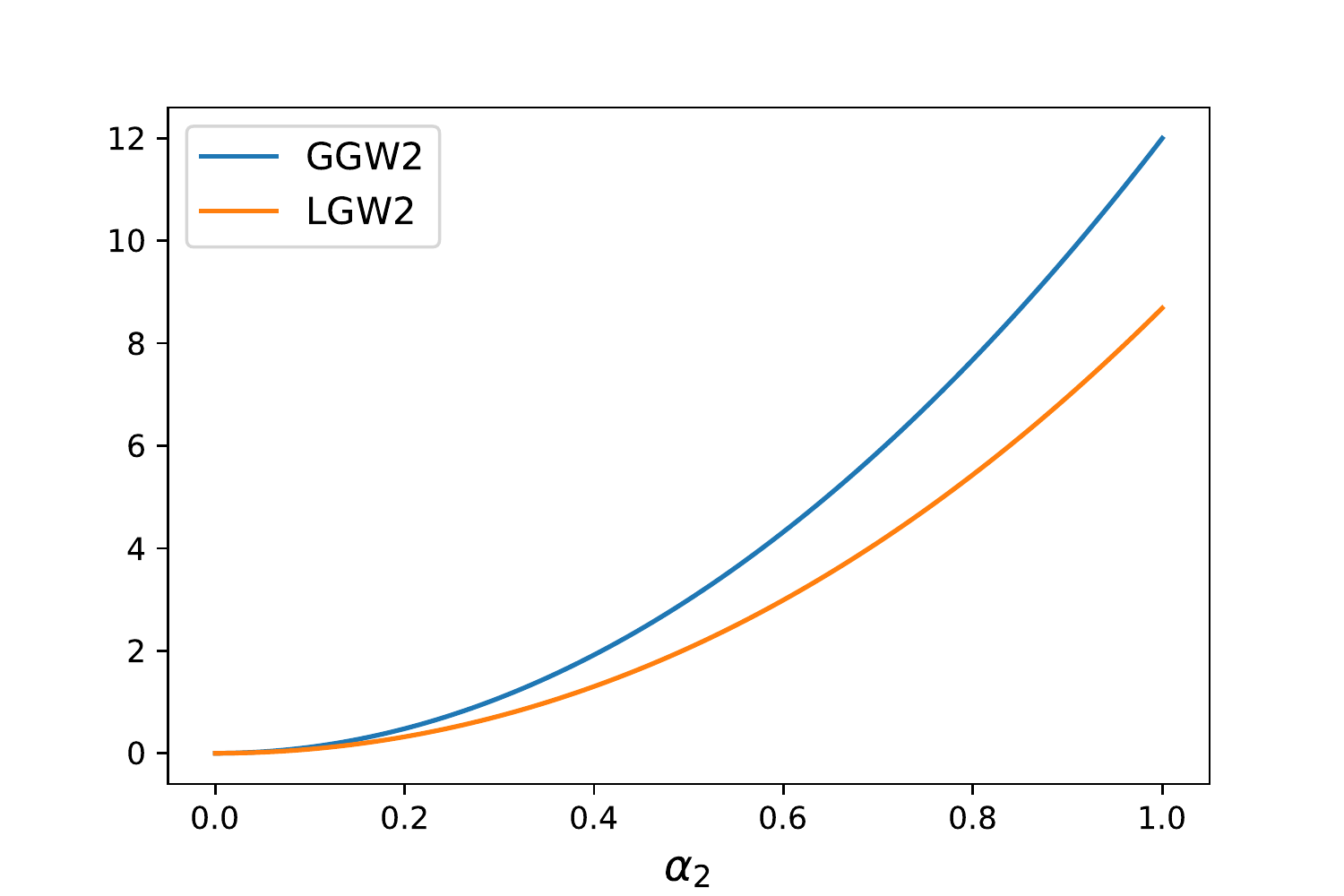}
        \end{subfigure}
        \begin{subfigure}[b]{0.32\textwidth}
            \includegraphics[width=\textwidth]{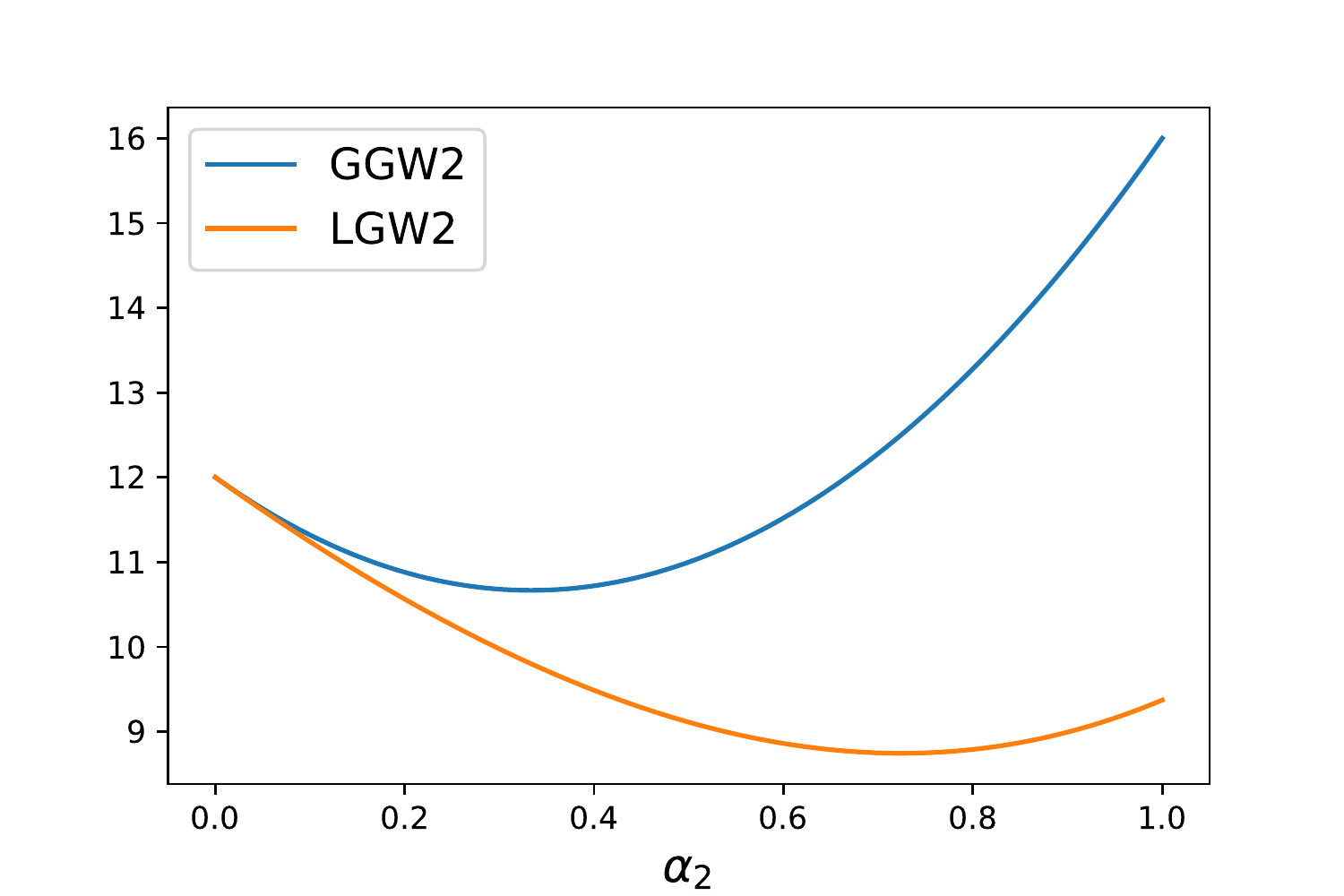}
        \end{subfigure}
        \begin{subfigure}[b]{0.32\textwidth}
            \includegraphics[width=\textwidth]{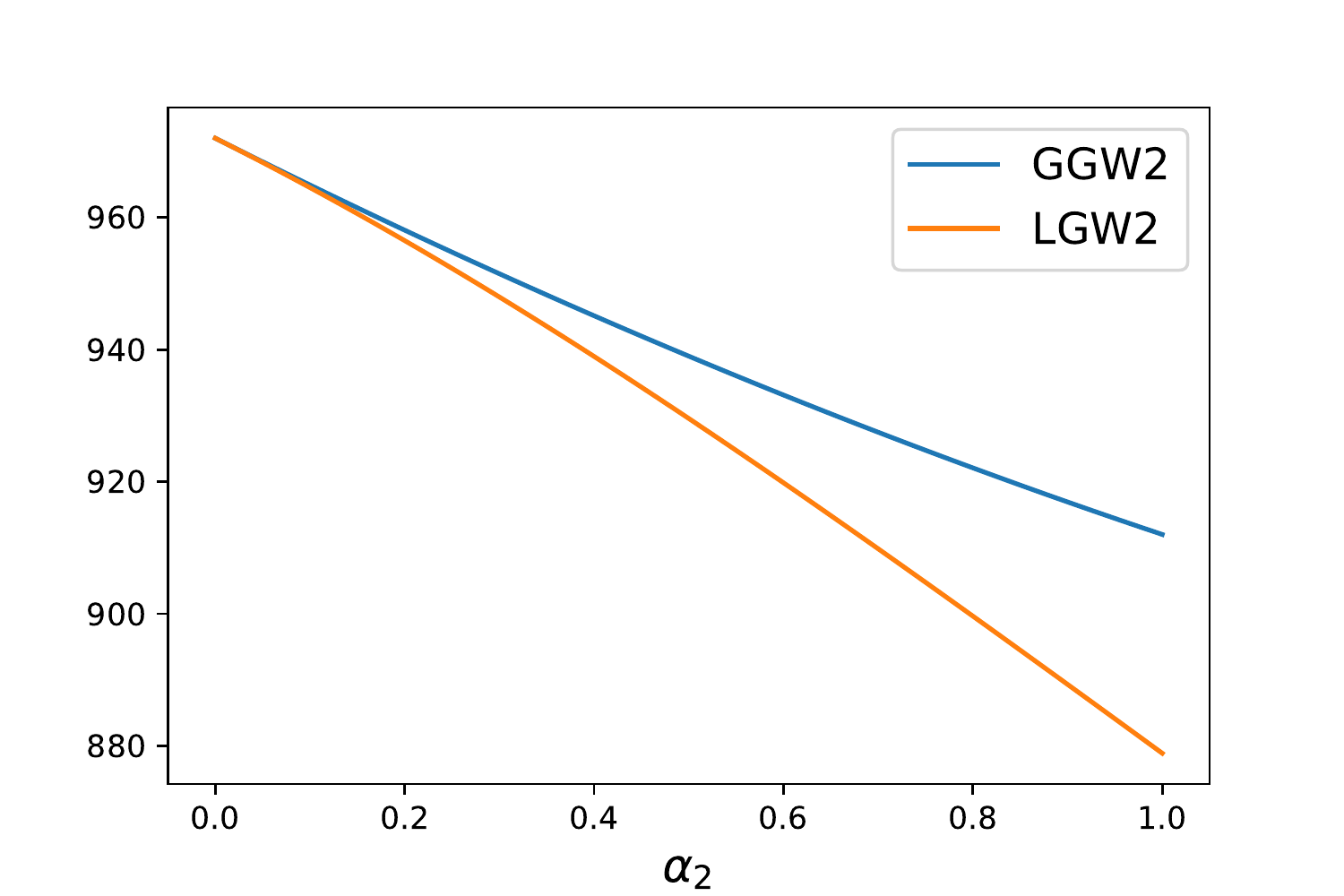}
        \end{subfigure}
    \caption{plot of $ GGW_2^2(\m,\n) $ and $ LGW_2^2(\m,\n) $ in function of $ \alpha_2 $ for $ \m = \mathcal{N}(0,\text{diag}(\alpha)) $, $ \n = \mathcal{N}(0,\beta_1) $, $ \alpha = (\alpha_1,\alpha_2)^T $, for 
    $(\alpha_1,\beta_1) = (1,1) $ (left), $ (\alpha_1,\beta_1) = (1,2) $ (middle), $ (\alpha_1,\beta_1) = (1,10) $ (right). One can easily compute using \eqref{eq:GGW} and \eqref{eq:LGW}
    that $ GGW_2^2(\m,\n) = 12\alpha_2^2 + 8\alpha_2(\alpha_1 - \beta_1) + 12(\alpha_1 - \beta_1)^2 $ and $ LGW_2^2(\m,\n) = 12\alpha_2^2 + 8\alpha_2(\alpha_1 - \beta_1) - 4\sqrt{\alpha_2^2 + \alpha_1^2}\beta_1 + 12(\alpha_1 - \beta_1)^2 + 8\alpha_1\beta_1$.}
    \label{fig:expe3}
    \end{figure}

\subsection{Explicit case}
As seen before, the difference between $ GGW_2^2(\m,\n) $ and $ LGW_2^2(\m,\n) $, with $ \m = \mathcal{N}(m_0,\Sigma_0) $ and $ \n = \mathcal{N}(m_1,\cov_1) $, is null when the two vectors of eigenvalues of $ \cov_0 $ and
$\cov_1 $ (sorted in decreasing order) are collinear. When we suppose $ \cov_0 $ non-singular, it implies that $ m = n $ and that the eigenvalues of $ \cov_1 $ are proportional to the eigenvalues of $ \cov_0 $ (rescaling). This case includes the more particular case where $ m = n = 1$. In that case $ \m = \mathcal{N}(m_0,\sigma_0^2) $
and $ \n = \mathcal{N}(m_1,\sigma_1^2) $, because $ \sigma_1 $ is always proportional to $ \sigma_0 $.

\begin{prop}
Suppose $ m = n $. Let $ \m = \mathcal{N}(m_0,\cov_0) $ and $ \n = \mathcal{N}(m_1,\cov_1) $ two Gaussian measures on $ \reel^m = \reel^n $. 
Let $ P_0,D_0 $ and $ P_1,D_1 $ be the respective diagonalizations of  $\cov_0 (= P_0D_0P_0^T)$ and $ \cov_1 (=P_1D_1P_1^T) $ which sort eigenvalues in non-increasing order.
Suppose $ \cov_0 $ is non-singular and that it exists a scalar $ \lambda \geq 0 $ such that $ D_1 = \lambda D_0 $. 
In that case, $ GW_2^2(\m,\n) = GGW_2^2(\m,\n) = LGW_2^2(\m,\n) $ and the problem admits a solution of the form $ (I_m, T)\#\m$ with T affine of the form:
\begin{equation}\label{eq:Tparticuliar}
\forall x \in \reel^m, \ T(x) = m_1 + \sqrt{\lambda} P_1\tilde{I}_mP_0^T(x - m_0),
\end{equation}
where $ \tilde{I}_m $ of the form $ \textup{diag}((\pm 1)_{i \leq m}) $. Moreover
\begin{equation}\label{eq:GW2particuliar}
GW_2^2(\m,\n) = (\lambda - 1)^2\left(4(\textup{tr}(\cov_0))^2 + 8\|\cov_0\|_{\spa{F}}^2\right).
\end{equation}
\end{prop}

\begin{proof} 
From \eqref{eq:cauchy}, we have
\begin{equation}
GGW_2^2(\m,\n) - LGW2^2(\m,\n) = 8\left(\|D_0\|_{\spa{F}}\|D_1\|_{\spa{F}} - \text{tr}(D_0D_1)\right).
\end{equation}
Noting $ \alpha \in \reel^m $ and $ \beta \in \reel^m $ the eigenvalues vectors of $ D_0 $ and $ D_1 $, it comes
\begin{equation}
GGW_2^2(\m,\n) - LGW2^2(\m,\n) = 8(\|\alpha\|\|\beta\| - \alpha^T\beta).
\end{equation}
Since it exists $ \lambda \geq 0 $ such that $ D_1 = \lambda D_0 $, we have $ \beta = \lambda \alpha $, and so
$ \alpha^T\beta = \|\alpha\|\|\beta\| $. Thus $ GGW_2^2(\m,\n) - LGW_2^2(\m,\n) = 0 $ and using Proposition \ref{prop:frame}, we get that $ GW_2^2(\m,\n) = GGW_2^2(\m,\n) = LGW^2_2(\m,\n) $. We get
\eqref{eq:Tparticuliar} and \eqref{eq:GW2particuliar} by simply reinjecting in \eqref{eq:optimalmapgromov} and \eqref{eq:GGW}.
\end{proof}

\begin{coro}
Let $ \m = \mathcal{N}(m_0,\sigma_0^2) $ and $ \n = \mathcal{N}(m_1,\sigma_1^2) $ be two Gaussian measures on $ \mathbb{R} $. Then 
\begin{equation}
GW_2^2(\m, \n) =  12\left(\sigma_0^2 - \sigma_1^2\right)^2,
\end{equation}
and the optimal transport plan $ \op^* $ has the form $ (I_1,T)\#\m $ with $ T $ affine of the form:
\begin{equation}\
\forall x \in \reel, \ T(x) = m_1 \pm \frac{\sigma_1}{\sigma_0}(x - m_0).
\end{equation}
Thus, the solution of $ W^2_2(\m,\n) $ is also solution of $ GW_2^2(\m,\n) $.
\end{coro}

\subsection{Case of degenerate measures}
In all the results exposed above, we have supposed $ \cov_0 $ non-singular, which means that $ \m $ is not degenerate. Yet, if $ \cov_0 $ is not full rank, one can easily extend the previous results thanks to the following proposition.

\begin{prop}\label{prop:singular}
Let $ \m = \mathcal{N}(0,D_0) $ and $ \n = \mathcal{N}(0,D_1) $ be two centered Gaussian measures on $ \reel^m $ and $ \reel^n $ with diagonal covariance matrices $ D_0 $ and $ D_1 $ with eigenvalues in decreasing order.
We denote $ r = \text{rk}(D_0)$ the rank of $ D_0 $ and we suppose that $ r < m $. Let us define $ P_r = \begin{pmatrix} I_r & 0_{r,m-r} \end{pmatrix} \in \reel^{r \times m} $. Then $ GW_2^2(\m,\n) = GW_2^2(P_r\#\m,\n) $, $ GGW_2^2(\m,\n) = GGW_2^2(P_r\#\m,\n)$, and $ LGW_2^2(\m,\n) = LGW_2^2(P_r\#\m,\n) $.
\end{prop}

\begin{proof}
For $ r < m $, we denote $ \Gamma_r(\reel^m) $ the set of vectors  $ x = (x_1,\dots,x_m)^T $ of $ \reel^m $ such that $ x_{r + 1} = \dots = x_{m} = 0 $. For $ \op \in \Pi(\m,\n) $, one can remark that for any borel set $ A \subset \reel^m \smallsetminus  \Gamma_r(\reel^m) $,
and any borel set
$ B \subset \reel^n $, we have $ \op(A,B) = 0 $ and so 
\begin{equation} 
\begin{split}
GW_2^2(\m,\n) &= \inf_{\op \in \Pi(\m,\n)}\int_{\reel^m \times \reel^n}\int_{\reel^m \times \reel^n}(\|x -x'\|^2_{\reel^m} - \|y -y'\|^2_{\reel^m})^2d\op(x,y)d\op(x',y') \\
& =\inf_{\op \in \Pi(\m,\n)}\int_{\Gamma_r(\reel^m) \times \reel^n}\int_{\Gamma_r(\reel^m) \times \reel^n}(\|x -x'\|^2_{\reel^m} - \|y -y'\|^2_{\reel^m})^2d\op(x,y)d\op(x',y') \\
&=\inf_{\op \in \Pi(\m,\n)}\int_{\Gamma_r(\reel^m)\times\reel^n}\int_{\Gamma_r(\reel^m)\times\reel^n}(\|P_r(x -x')\|^2_{\reel^r} - \|y -y'\|^2_{\reel^m})^2d\op(x,y)d\op(x',y') 
\end{split}
\end{equation}
Now, observe that for $\op \in \Pi(\m,\n)$, $(P_r,I_n)\#\op \in \Pi(P_r\#\m,\n)$. It follows that
\begin{equation} 
\begin{split}
GW_2^2(\m,\n) &\le \inf_{\op \in \Pi(P_r\#\m,\n)}\int_{\reel^r\times\reel^n}\int_{\reel^r\times\reel^n}(\|x -x'\|^2_{\reel^r} - \|y -y'\|^2_{\reel^m})^2d\op(x,y)d\op(x',y') \\
&= GW_2^2(P_r\#\m,\n).
\end{split}
\end{equation}
Conversely, since $\m$ has no mass outside of $\Gamma_r(\reel^m)$, $P_r^T\#P_r\#\m = \m$, which implies that for $\op \in \Pi(P_r\#\m,\n)$, $(P_r^T,I_n)\#\op \in \Pi(\m,\n)$. It follows that \begin{equation} 
\begin{split}
GW_2^2(P_r\#\m,\n)& =  \inf_{\op \in \Pi(P_r\#\m,\n)}\int_{\reel^r\times\reel^n}\int_{\reel^r\times\reel^n}(\|x -x'\|^2_{\reel^r} - \|y -y'\|^2_{\reel^m})^2d\op(x,y)d\op(x',y') \\
&=  \inf_{\op \in \Pi(P_r\#\m,\n)}\int_{\reel^r\times\reel^n}\int_{\reel^r\times\reel^n}(\|P_r^T(x -x')\|^2_{\reel^r} - \|y -y'\|^2_{\reel^m})^2d\op(x,y)d\op(x',y') \\
&\le\inf_{\op \in \Pi(\m,\n)}\int_{\reel^m\times\reel^n}\int_{\reel^m\times\reel^n}(\|x -x'\|^2_{\reel^r} - \|y -y'\|^2_{\reel^m})^2d\op(x,y)d\op(x',y') \\
&\le GW_2^2(\m,\n) .
\end{split}
\end{equation}

The exact same reasoning can be made in the case of $ GGW_2 $. Morover, it can be easily seen when looking at \eqref{eq:LGW} that $ LGW_2^2(\m,\n) = LGW_2^2(P_r\#\m,\n) $. 
\end{proof}

Thus, when $ \cov_0 $ is not full rank, one can apply Proposition \ref{prop:singular} and consider directly the Gromov-Wasserstein distance between the projected (non-degenerate) measure $ P_r\#\m $ on $ \reel^r $ and $ \n $ and so Proposition \ref{prop:frame} still holds when $ \m $ is degenerate.

In the case of $ GGW_2 $, an explicit optimal transport plan can still be exhibited. In the following, we denote $  r_0  $ and $ r_1 $ the ranks of $ \cov_0 $ and $ \cov_1 $, and we suppose without loss of generality 
that $ r_0 \geq r_1 $, but this time not necessarily that $ m \geq n $.  If $ \m = \mathcal{N}(m_0,\cov_0) $ and $ \n = \mathcal{N}(m_1,\cov_1) $ are two Gaussian measures on $ \reel^m $ and $ \reel^n $, and $ (P_0,D_0) $ and $ (P_1,D_1) $
are the respective diagonalizations of $ \cov_0 (= P_0D_0P_0^T) $ and $ \cov_1 (= P_1D_1P_1^T) $ which sort the eigenvalues in decreasing order, an optimal transport plan which achieves $ GGW_2(\m,\n) $ is of the form $ \op^* = (I_m,T)\#\m $ with 
\begin{equation}
\forall x \in \reel^m, T(x) = m_1 + P_1AP_0^T(x - m_0),
\end{equation}
where $ A \in \reel^{n \times m} $ is of the form
$$ A = \begin{pmatrix}  \tilde{I}_{r_1}(D_1^{(r_1)})^{\frac{1}{2}}(D_0^{(r_1)})^{-\frac{1}{2}} & 0_{r_1,m - r_1} \\ 0_{n - r_1, r_1}  & 0_{n - r_1, m - r_1} \end{pmatrix}, $$ 
where $ \tilde{I}_{r_1} $ is any matrix of the form $ \text{diag}((\pm 1)_{i \leq r_1}) $. 

\section{Behavior of the empirical solution}\label{sec:6}

\begin{figure}[!h]
    \centering
    \begin{subfigure}[b]{0.49\textwidth}
        \includegraphics[width=\textwidth]{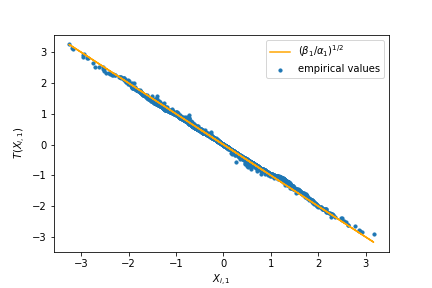}
    \end{subfigure}
    \begin{subfigure}[b]{0.49\textwidth}
        \includegraphics[width=\textwidth]{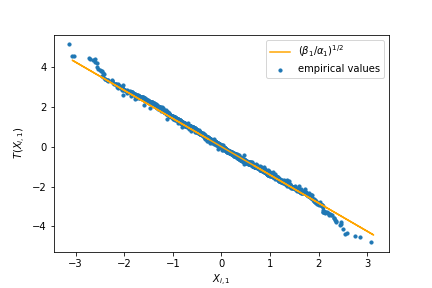}
    \end{subfigure}
    \begin{subfigure}[b]{0.49\textwidth}
        \includegraphics[width=\textwidth]{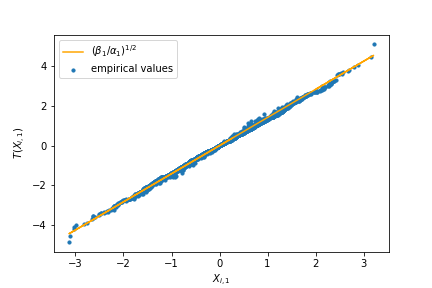}
    \end{subfigure}
    \begin{subfigure}[b]{0.49\textwidth}
        \includegraphics[width=\textwidth]{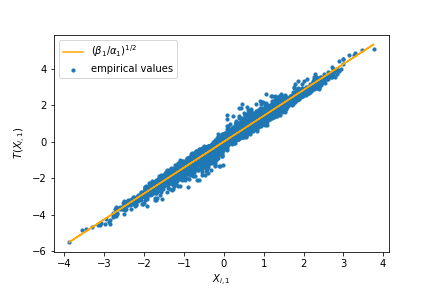}
    \end{subfigure}
    \begin{subfigure}[b]{0.49\textwidth}
        \includegraphics[width=\textwidth]{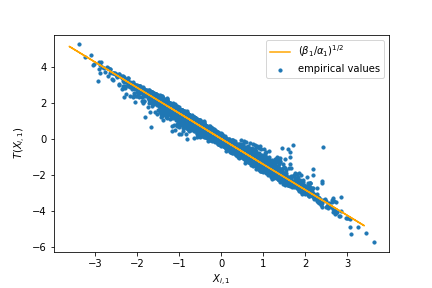}
    \end{subfigure}
    \begin{subfigure}[b]{0.49\textwidth}
        \includegraphics[width=\textwidth]{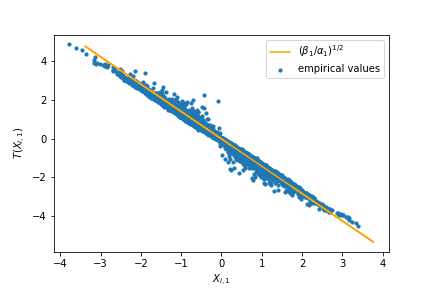}
    \end{subfigure}
    \caption{plot of the first coordinate of samples $ Y_i $ in fonction of the
    the first coordinate of their assigned samples $ X_j $  (blue dots) and line of equation $ y = \pm\sqrt{\beta}x $  (orange line) for $ k = 2000 $, $ \alpha = (1,0.1)^T $ and $ \beta = 2 $ (top left), $ k = 2000 $, $ \alpha = (1,0.1)^T $ and $ \beta = (2,0.3)^T $ (top right), $ k = 2000 $, $ \alpha = (1,0.1,0.01)^T $ and 
    $ \beta = 2 $ (middle left),
    $ k = 7000 $, $ \alpha = (1,0.3)  $ and $\beta = 2$ (middle right), $ k = 7000 $, $ \alpha = (1,0.1)^T $, and $ \beta = (2,1)^T $ (bottom left), and $ k = 7000 $ and $ \alpha = (1,0.3,0.1)$ and $ \beta = 2 $
    (bottom right).}
    \label{fig:expe4}
\end{figure}

To complete the previous study, we perform a simple experiment to illustrate the  behavior of the  solution of the Gromov Wasserstein problem.    
In this experiment, we draw independently $ k $ samples $ (X_j)_{j \leq k} $ and $ (Y_i)_{i \leq k} $ from respectively $ \m = \mathcal{N}(0,\text{diag}(\alpha))$  and $ \n = \mathcal{N}(0,\text{diag}(\beta)) $ with 
$ \alpha \in \reel^m $ and $ \beta \in \reel^n $. Then we compute the Gromov-Wasserstein distance between the two histograms $ X $ and  $ Y $ with the algorithm proposed in \cite{algo} using the 
Python Optimal Transport library \footnote{The library is accessible here: https://pythonot.github.io/index.html}. In Figure \ref{fig:expe4}, we plot the first coordinates of the samples $ Y_i $ in fonction of the
the first coordinate of the samples $ X_j $ they have been assigned to by the algorithm (blue dots). We draw also the line of equation $ y = \pm\sqrt{\beta}x $ to compare with the theorical solution 
of the Gaussian restricted problem (orange line) for $ k = 2000 $, $ \alpha = (1,0.1)^T $ and $ \beta = 2 $ (top left), $ k = 2000 $, $ \alpha = (1,0.1)^T $ and $ \beta = (2,0.3)^T $ (top right), $ k = 2000 $, $ \alpha = (1,0.1,0.01)^T $ and 
$ \beta = 2 $ (middle left),
$ k = 7000 $, $ \alpha = (1,0.3)  $ and $\beta = 2$ (middle right), $ k = 7000 $, $ \alpha = (1,0.1)^T $, and $ \beta = (2,1)^T $ (bottom left), and $ k = 7000 $ and $ \alpha = (1,0.3,0.1)$ and $ \beta = 2 $
(bottom right). Observe that the empirical solution seems to be behaving exactly in the same way as the theoretical solution exhibited in theorem \ref{thm:gaussian} as soon as $ \alpha $ and $ \beta $ are close to be collinear. However, when $ \alpha $ and $ \beta $ are further away from  collinearity, determining the behavior of the empirical solution becomes more complex. Solving Gromov-Wasserstein numerically, even approximately, is a particularly hard task, therefore  we cannot conclude if the empirical solution does not behave in 
the same way as the theorical solution exhibited in theorem \ref{thm:gaussian} or if the algorithm has not converged in these more complex cases. This second assumption seems to be more likely because it seems that increasing the number of points $ k $ 
reduces the gap between the blue dots and the orange line. Thus, we conjecture that the optimal plan which achieves $ GGW_2(\m,\n) $ is also solution of the non-restricted problem $ GW_2(\m,\n) $ and that $ GW_2(\m,\n) = GGW_2(\m,\n) $.

\section{Conclusion}

In this paper, we have exhibited lower and upper bounds for the Gromov-Wasserstein distance (with a squared ground distance)  between Gaussian measures living on different Euclidean spaces. We have also studied the tightness of the provided bounds, both theoretically and numerically.   The upper bound is obtained through the study of the problem with the additional restriction that the optimal plan itself is Gaussian. We have shown that this particular case has  a very simple closed-form solution, which can be described as first performing PCA on both distributions and then deriving the optimal linear plan between these aligned distributions.  We conjecture that the linear solution exhibited when adding this restriction might also be the solution in more general cases.

\section{Appendix: proof of the lemmas}\label{sec:appendix}
\subsection{Proof of Lemma \ref{lemme:maxnorm2}}

\begin{proof}
The proof is inspired from the proof of Equation \eqref{eq:w2gaussian} provided in \cite{OTGaussian}. We want to maximize $ \text{tr}(K^TK) $ with the constraint that $ \Sigma $ is semi-definite positive. Let $ S = \cov_1 - K^T\cov_0^{-1}K $ (Schur complement). Problem \eqref{eq:maxnorm2} can be written in the following way
\begin{equation}\label{eq:order2matrix}
\min_{S \in S^+_n(\mathbb{R}) } - \text{tr}(K^TK).
\end{equation}
For a given $ S $, the set of feasible $ K $ is the set of $ K $ such that $ K^T\cov_0^{-1}K = \cov_1 - S $. Since $ \cov_0 \in S^{++}_m(\mathbb{R}) $,  $ K^T\cov_0^{-1}K \in S^{+}_n(\mathbb{R}) $ and so $ \cov_1 - S \in S^{+}_n(\mathbb{R}) $. We note $ r $ the rank of $ K^T\cov_0^{-1}K $. One can observe that $$ r \leq n \leq m, $$
where the left-hand side inequality comes from the fact that $ \text{rk}(AB) \leq \min\{\text{rk}(A),\text{rk}(B)\}$. Then,  $ \cov_1 - S $ can be diagonalized
\begin{equation}\label{eq:diagonalisation}
\cov_1 - S = K^T\cov_0^{-1}K = U\Lambda^2 U^T = U_r\Lambda^2_rU_r^T,
\end{equation}
with  $ \Lambda^2 = \text{diag}(\lambda_1^2,...,\lambda_r^2,0,...,0)$, $ \Lambda_r^2 = \text{diag}(\lambda_1^2,...,\lambda_r^2) $, and $ U_r \in \mathbb{V}_r(\mathbb{R}^n) := \{M \in \mathbb{R}^{n \times r} | M^TM = I_r \} $ (Stiefel Manifold \cite{james1976topology}) such that  $ U = \begin{pmatrix} U_r  & U_{n-r} \end{pmatrix}$ . 
From \eqref{eq:diagonalisation}, we can deduce that
\begin{equation}
(\cov_0^{-\frac{1}{2}}KU_r\Lambda_r^{-1})^T\cov_0^{-\frac{1}{2}}KU_r\Lambda_r^{-1} = I_r. 
\end{equation}
We can set $ B_r = \cov_0^{-\frac{1}{2}}KU_r\Lambda_r^{-1} $ such that $ B_r \in 
\mathbb{V}_r(\mathbb{R}^m) $. One can deduce that
\begin{equation}
\nonumber KU_r = \cov_0^{\frac{1}{2}}B_r\Lambda_r.
\end{equation}
Moreover, since $ U_{m-r}^TK^T\cov_0^{-1}KU_{m-r} = 0 $ and $\cov_0 \in S_m^{++}(\mathbb{R}) $, it comes that $ KU_{n-r} = 0 $ and so 
\begin{equation}\label{eq:Kexpression}
K = KUU^T = KU_rU_r^T = \cov_0^{\frac{1}{2}}B_r\Lambda_rU_r^T.
\end{equation}
We can write $ \text{tr}(K^TK) $ as a function of $ B_r$ :
\begin{equation}
\begin{split}
\text{tr}(K^TK) &=\text{tr}(U_r\Lambda_rB_r^T\cov_0B_r\Lambda_rU_r^T) \\ &= \text{tr}(U_r^TU_r\Lambda_rB_r^T\cov_0B_r\Lambda_r) \\
&= \text{tr}(\Lambda^2_rB_r^T\cov_0B_r),
\end{split}
\end{equation}
Thus, for a given $ S $, the set of $ K $ such that $ K^T\cov_0^{-1}K = \cov_1 - S $ is parametrized by the $r$-frame $ B_r $. We want to find $ B_r $ which maximizes $ \text{tr}(K^TK) $ for a given $ S $. This problem can be rewritten:
\begin{equation}\label{eq:QQP}
    \min_{B_r \in \mathbb{V}_r(\mathbb{R}^m)} - \text{tr}(\Lambda^2_rB_r^T\cov_0B_r).
\end{equation}
The following is a readaptation of the proof of the Proposition (3.1) in \cite{quadratic} when $ B_r $ is not a squared matrix. The Lagrangian of problem \eqref{eq:QQP} can be written
\begin{equation}
\nonumber \mathcal{L}(B_r,C) = -\text{tr}(\Lambda^2_rB_r^T\cov_0B_r) + \text{tr}(C(B_r^TB_r - I_r)),
\end{equation} 
where $ C \in S_r(\mathbb{R})  $ is the Lagrange multiplier associated to the constraint $ B_r^TB_r = I_r $ ($C$ is symmetric because $ B_r^TB_r - I_r $ is symmetric). We can then derive the first-order condition

\begin{equation}
\nonumber -2\cov_0B_r\Lambda_r^2 + 2B_rC = 0,
\end{equation}
or equivalently
\begin{equation}
\cov_0B_r\Lambda_r^2B_r^T = B_rCB_r^T.
\end{equation}
Since $ C \in S_r(\mathbb{R})  $, $ B_rCB_r^T \in S_m(\mathbb{R}) $ and $ \cov_0B_r\Lambda_r^2B_r^T \in S_m(\mathbb{R}) $. We can deduce that $ \cov_0 $ and $ B_r\Lambda_rB_r^T $ commute. Moreover, since $ \cov_0 $ and $ B_r\Lambda_r^2B_r^T $ are both symmetric,  they can be diagonalized in the same basis. Since $ B_r \in \mathbb{V}_r(\mathbb{R}^m)  $, it can be seen as the $ r $ first vectors of an orthogonal basis of $ \mathbb{R}^m $. It means there exists a matrix $ B_{m-r} $ such that
\begin{equation}
\nonumber B_r\Lambda_r^2B_r^T = B\Lambda_m^2B^T,
\end{equation}
where $ \Lambda^2_m \in \mathbb{R}^{m \times m} = \text{diag}(\lambda_1^2,...,\lambda_r^2,0,...,0) $ and $ B = \begin{pmatrix} B_r & B_{m-r} \end{pmatrix} $. Thus the eigenvalues of  $ B_r\Lambda_r^2B_r^T $ are exactly the eigenvalues of $ \Lambda^2_m $. Since $ \cov_0 $ and $ B_r\Lambda_r^2B_r^T $ can be diagonalized in the same basis, we get that $ \text{tr}(\Lambda_r^2B_r^T\cov_0B_r) = \text{tr}(\cov_0B_r\Lambda_r^2B_r^T) = \text{tr}(D_0\tilde{\Lambda}_m)$ where $ \tilde{\Lambda}_m $ is a diagonal matrix with the same eigenvalues as $ \Lambda_m $, but in a different order. Now, it can be easily seen that the optimal value of \eqref{eq:QQP} is reached when $B_r $ is a permutation matrix which sorts the eigenvalues of $ \Lambda_m $ in decrasing order.

Thus, for a given $S$, the maximum value of $ \text{tr}(K^TK) $ is 
$ \text{tr}(D_0\tilde{\Lambda}_m(S))$. We can now establish for which $ S,  \ \text{tr}(D_0\tilde{\Lambda}_m(S)) $ is optimal. For a given $ S $, we denote $ \lambda_1, ..., \lambda_n $ the eigenvalues of $ \cov_1 - S $ and $ \beta_1,...,\beta_n $ the eigenvalues of $ \cov_1 $ ordered in decreasing order. Since $ S \in S_n^+(\mathbb{R}) $,  $ \forall x \in \mathbb{R}^n $, the following inequality holds: 
\begin{equation}
x^T(\cov_1 - S)x \leq x^T\cov_1x.
\end{equation}
This inequality still holds when restricted to any subspace of $ \reel^n $. Using the Courant-Fischer theorem, we can conclude that:
\begin{equation}
\forall i \leq n, \ \lambda_i \leq \beta_i.
\end{equation}
Thus, the optimal value of $ \text{tr}(D_0\tilde{\Lambda}_m(S)) $ is reached when $ S = 0 $ and $ \tilde{\Lambda}_m(0) = \begin{pmatrix} D_1 & 0 \\ 0 & 0 \end{pmatrix} $ and so $ \text{tr}(D_0\tilde{\Lambda}_m(0)) = \text{tr}(D_0^{(n)}D_1) $. Let $ A = \begin{pmatrix} \tilde{I}_n(D_0^{(n)})^\frac{1}{2}D_1^\frac{1}{2} \\ 0_{m-n,n} \end{pmatrix} $ with $\tilde{I}_n $ of the form $ \text{diag}((\pm 1)_{i \leq n})) $. It can be easily verified that  $ A^TD_0^{-1}A = D_1 $ and if $ K^* = P_0^TAP_1  $, $ K^{*T}\cov_0^{-1}K^{*} = P_1^TA^TP_0\cov_0^{-1}P_0^TAP_1 = P_1^TA^TD_0^{-1}AP_1 = P_1^TD_1P_1 = \cov_1 $ and $ K^{*T}K^* $ has the same eigenvalues as $ A^TA $ and  $ \text{tr}(A^TA) = \text{tr}(D_0^{(n)}D_1) $.
\end{proof}

\subsection{Proof of Lemma \ref{lemme:maxnorm1}}
In order to prove lemma \ref{lemme:maxnorm1}, we will use the following lemma, demonstrated by Antreicher and Wolkowicz \cite{quadratic}.

\begin{lemma}\label{lemme:Antreicher}
    \textbf{\textup{(Anstreicher and Wolkowicz, 1998, \cite{quadratic})}}
    Let $\Sigma_0 $ and  $\Sigma_1 $ be two symmetric matrices of size $ n $. We note $ \Sigma_0 = P_0\Lambda_0P_0^T $ and $ \Sigma_1 = P_1\Lambda_1P_1^T $ there respective diagonalization such that the eigenvalues of $ \Lambda_0 $ are sorted in non-increasing order and the eigenvalues of $ \Lambda_1 $ are sorted in increasing order. Then
    \begin{equation}
    \min_{PP^T = I_n} \textup{tr}(\Sigma_0P\Sigma_1P^T) = \textup{tr}(\Lambda_0\Lambda_1),
    \end{equation}
    and it is achieved for $ P^* = P_0P_1^T $.
\end{lemma}

\begin{proof}[Proof of Lemma \ref{lemme:maxnorm1}]
We proceed in the same way as before: first, we derive the expression of the optimal value for a given $ S = \cov_1 - K^T\cov_0^{-1}K $, then we determine for which $ S $ this expression is maximum. The start of the proof is exactly the same as the proof of \eqref{lemme:maxnorm2} until formula \eqref{eq:Kexpression}. We diagonalize $ \cov_1 - S = K^T\cov_0^{-1}K = U_r\Lambda_rU_r^T $ where $ r$ is the rank of $ K^T\cov_0^{-1}K$, then we set $ B_r  = \cov_0^{-\frac{1}{2}}KU_r\Lambda_r^{-1}$ while observing that $ B_r \in \mathbb{V}_r(\mathbb{R}^m) $ and we deduce that $ K = \cov_0^{\frac{1}{2}}B_r\Lambda_rU_r^T $. By reinjecting this expression, it comes that
\begin{equation} \text{tr}(KA) = \text{tr}(A^TK^T) = \text{tr}(A^TU_r\Lambda_rB_r^T\cov_0^{\frac{1}{2}}) = \text{tr}(\cov_0^{\frac{1}{2}}A^TU_r\Lambda_rB_r^T). \end{equation}
For a given $ S $, the problem of finding the optimal value is parametrized by $ B_r $ and is:
\begin{equation}
\min_{B_r \in \mathbb{V}_r(\mathbb{R}^m)} - \text{tr}(\cov_0^{\frac{1}{2}}A^TU_r\Lambda_rB_r^T).
\end{equation}
The Lagrangian of this problem can be written:
\begin{equation}
\mathcal{L}(B_r,C) = -\text{tr}(\cov_0^{\frac{1}{2}}A^TU_r\Lambda_rB_r^T) + \text{tr}(C(B_r^TB_r - I_r)),
\end{equation} 
where $ C \in S_r(\mathbb{R})  $ is the Lagrangian multiplier associated to the constraint $ B_r^TB_r = I_r $. We can then derive the first-order condition:
\begin{equation}
\nonumber -\cov_0^{\frac{1}{2}}A^TU_r\Lambda_r + 2B_rC = 0,
\end{equation}
or equivalently:
\begin{equation}
\nonumber \cov_0^{\frac{1}{2}}A^TU_r\Lambda_rB_r^T =  2B_rCB_r^T.
\end{equation}
Since $ C \in S_r(\mathbb{R}) $, $ 2B_rCB_r^T \in S_m(\mathbb{R}) $ and $ \cov_0^{\frac{1}{2}}A^TU_r\Lambda_rB_r^T \in S_m(\mathbb{R}) $. Moreover, the rank of $ \cov_0^{\frac{1}{2}}A^TU_r\Lambda_rB_r^T $ is equal to $ 1 $ because $ \text{rk}(A) = 1 $ and $ \text{rk}(\cov_0^{\frac{1}{2}}A^TU_r\Lambda_rB_r^T) = 0 $ would imply that $ \text{tr}(KA) = 0 $, which cannot be the maximum value of our problem. So there exists a vector $ u_m \in \mathbb{R}^m $ such that  
\begin{equation}
\cov_0^{\frac{1}{2}}A^TU_r\Lambda_rB_r^T = u_mu_m^T.
\end{equation}
    Then we can reinject the value $ B_r $ in the expression:
    \begin{equation}
    \begin{split}
    \cov_0^{\frac{1}{2}}A^TU_r\Lambda_rB_r^T &= \cov_0^{\frac{1}{2}}A^TU_r\Lambda_r\Lambda_r^{-1}U_r^TK^T\cov_0^{-\frac{1}{2}} \\
    &= \cov_0^{\frac{1}{2}}A^TU_rU_r^TK^T\cov_0^{-\frac{1}{2}} \\
    &= \cov_0^{\frac{1}{2}}A^TK^T\cov_0^{-\frac{1}{2}},
    \end{split}
    \end{equation}
    where we used the fact that $ K = KUU^T = KU_rU_r^T $ because $ KU_{n-r} = 0 $. We have so on one hand:
    \begin{equation}
    \text{tr}(KA) = \text{tr}(\cov_0^{\frac{1}{2}}A^TK^T\cov_0^{-\frac{1}{2}}) = \text{tr}(u_mu_m^T) = u_m^Tu_m,
    \end{equation}
    and on the other hand:
    \begin{equation}
    \begin{split}
    \cov_0^{\frac{1}{2}}A^TK^T\cov_0^{-\frac{1}{2}}(\cov_0^{\frac{1}{2}}A^TK^T\cov_0^{-\frac{1}{2}})^T &= \cov_0^{\frac{1}{2}}A^TK^T\cov_0^{-1}KAD_0^{\frac{1}{2}} \\ &= \cov_0^{\frac{1}{2}}A^T(\cov_1 - S)A\cov_0^{\frac{1}{2}} \\ &= u_mu_m^Tu_mu_m^T \\
    &= u_m^Tu_mu_mu_m^T,
    \end{split}
    \end{equation}
    and thus
    \begin{equation}
    \text{tr}(\cov_0^{\frac{1}{2}}A^T(\cov_1 - S)A\cov_0^{\frac{1}{2}}) = u_m^Tu_m\text{tr}(u_mu_m^T) = (u_m^Tu_m)^2 = (\text{tr}(KA))^2.
    \end{equation} 
    Then we will determine for which $ S $, $ \text{tr}(\cov_0^{\frac{1}{2}}A^T(\cov_1 - S)A\cov_0^{\frac{1}{2}}) $ is maximum:
    \begin{equation}
    \begin{split}
    \text{tr}(\cov_0^{\frac{1}{2}}A^T(\cov_1 - S)A\cov_0^{\frac{1}{2}}) &= \text{tr}(A\cov_0A^T(\cov_1 - S)) \\ &=  \text{tr}(A\cov_0A^T\cov_1) - \text{tr}(A\cov_0A^TS).
    \end{split}
    \end{equation}
    Let $ B = A\cov_0A^T $. We can observe that $ B \in S^+_n(\mathbb{R}) $ with rank $ 1 $. Moreover, since $ S \in S_n^+(\mathbb{R}) $, it can be diagonalized, and we denote $ S = PDP^T $. As before, we will first determine the value of $ \text{tr}(BS) $ for a given $ D $, then we will determine which $ D $ minimizes $ \text{tr}(BS) $. For a given $ D $, we want the optimal value of
    \begin{equation}
    \min_{PP^T = I_b} \text{tr}(BPDP^T).
    \end{equation}

    Since $ B $ is symmetric with rank 1, it has only one non null eigenvalue which is equal to its trace. Using Lemma \ref{lemme:Antreicher}, we can deduce that
    \begin{equation} 
    \min_{PP^T = I_n} \text{tr}(BPDP^T) = \text{tr}(B)\lambda_n, \end{equation}
    where $ \lambda_n $  is the smallest eigenvalue of $ D $. Since $ S \in S_n^+(\mathbb{R}) $, the smallest possible value for $ \lambda_n $ is $ 0 $. 
    
    If  $ \cov_0 = \textup{diag}(\alpha)$, $ \cov_1 = \textup{diag}(\beta)$, it can be easily seen that $ \text{tr}(\mathbb{1}_{n,m}\cov_0\mathbb{1}_{m,n}\cov_1)  = \text{tr}(\cov_0)\text{tr}(\cov_1)$. Thus,
    if $ K = \frac{\alpha\beta^T}{\sqrt{\text{tr}(\cov_0)\text{tr}(\cov_1)}} = \frac{\cov_0\mathbb{1}_{m,n}\cov_1}{\sqrt{\text{tr}(\cov_0)\text{tr}(\cov_1)}}  $, we can observe
    that 
    \begin{equation}
    \text{tr}(K\mathbb{1}_{n,m}) = \text{tr}(\mathbb{1}_{n,m}K) = \frac{\text{tr}(\mathbb{1}_{n,m}\cov_0\mathbb{1}_{m,n}\cov_1)}{\sqrt{\text{tr}(\cov_0)\text{tr}(\cov_1)}} = \sqrt{\text{tr}(\cov_0)\text{tr}(\cov_1)},
    \end{equation}
    Now we must show that $ S = \cov_1 - K^T\cov_0^{-1}K \in S_n^+(\reel) $. To do so, we will show that $ \forall i \leq n $, the determinant of the principal minor $S^{(i)}$ is positive. We can derive that
    \begin{equation}
    S = \cov_1 - \frac{\beta\alpha^T\cov_0^{-1}\alpha\beta^T}{\text{tr}(\cov_0)\text{tr}(\cov_1)} = \cov_1 - \frac{\beta\beta^T\text{tr}(\cov_0)}{\text{tr}(\cov_0)\text{tr}(\cov_1)} = \cov_1 - \frac{\beta\beta^T}{\text{tr}(\cov_1)}.
    \end{equation}
    Using the matrix determinant lemma, it comes that
    \begin{equation}
    \forall i \leq n, \ \text{det}(S^{(i)}) = \prod_k^i\beta_k\left(1 - \frac{\text{tr}(\cov_1^{(i)})}{\text{tr}(\cov_1)}\right).
    \end{equation}
    Thus, $ \forall i < n$, $ \text{det}(S^{(i)}) > 0 $, and $ \text{det}(S) = 0 $. We conclude that $ S \in S_n^+(\reel) $ and the smallest eigenvalue of $ S $ is 0.
    \end{proof}

\subsection{Proof of Lemma \ref{lemme:innerproduct}}
\begin{proof}
For $m\geq 1$, let $\Gamma_m$ denote the set of vectors $v=(v_1,\ldots, v_m)$ of $\reel^m$ such that $v_1\geq v_2 \geq \ldots \geq v_m\geq 0$ and $\sum_{i=1}^m v_i^2=1$. We want to prove that 
\begin{equation} 
\forall u, v \in \Gamma_m , \quad \sum_{i=1}^m u_i v_i \geq \frac{1}{\sqrt{m}} .
\end{equation}
We proceed by induction on $m$. For $m=1$, it's obviously true since $\Gamma_1=\{1\}$. Assume now $m>1$, and the result true for $m-1$.
Let $u,v\in\Gamma_m$, then using the result for $(u_2,\ldots,u_m)/(\sum_{i=2}^m u_i^2)^{1/2}$ and $(v_2,\ldots,v_m)/(\sum_{i=2}^m v_i^2)^{1/2}$ that both belong to $\Gamma_{m-1}$, we have
\begin{equation}
\begin{split}
  \sum_{i=1}^m u_i v_i  = u_1 v_1 + \sum_{i=2}^m u_i v_i & \geq  u_1 v_1 + \frac{1}{\sqrt{m-1}} \left( \sum_{i=2}^m u_i^2 \right)^{1/2} \left( \sum_{i=2}^m v_i^2 \right)^{1/2} \\
       &  \quad = u_1 v_1 + \frac{1}{\sqrt{m-1}} \sqrt{1-u_1^2}\sqrt{1-v_1^2} .
\end{split}
\end{equation}
Now since $u,v\in\Gamma_m$, we have $u_1,v_1\in[\frac{1}{\sqrt{m}},1]$. Let us denote
$F(u_1,v_1)=u_1 v_1 + \frac{1}{\sqrt{m-1}} \sqrt{1-u_1^2}\sqrt{1-v_1^2}$.
We have for all $v_1\in[\frac{1}{\sqrt{m}},1]$~:
\begin{equation}
F(1,v_1)=v_1\geq \frac{1}{\sqrt{m}} \quad \text{ and } \quad F(\frac{1}{\sqrt{m}},v_1)= \frac{\sqrt{1-v_1^2} + v_1}{\sqrt{m}} \geq \frac{1-v_1^2 + v_1}{\sqrt{m}} \geq \frac{1}{\sqrt{m}} .
\end{equation}
And computing the partial derivative of $F$ with respect to $u_1$, we get
\begin{equation}
\frac{\partial F}{\partial u_1}(u_1,v_1)= v_1 - \frac{u_1\sqrt{1-v_1^2}}{\sqrt{m-1}\sqrt{1-u_1^2}} .
\end{equation}
This is a decreasing function of $u_1$, with value $v_1$ at $u_1=0$ and value that goes to $-\infty$ when $u_1$ goes to $1$. Therefore the function $F(\cdot,v_1)$ on $[0,1]$  is first increasing and then decrasing, showing that
\begin{equation}
\forall u_1\in  [\frac{1}{\sqrt{m}},1], \quad F(u_1,v_1) \geq \min \left( F(\frac{1}{\sqrt{m}},v_1), F(1,v_1) \right) \geq \frac{1}{\sqrt{m}} .
\end{equation}
Finally we thus have proved that 
$$ \sum_{i=1}^m u_i v_i \geq \frac{1}{\sqrt{m}} , $$
and moreover the equality is achieved when the vectors $u$ and $v$ are the vectors $(1,0,\ldots,0)$ and $(\frac{1}{\sqrt{m}}, \frac{1}{\sqrt{m}},\ldots, \frac{1}{\sqrt{m}})$.
\end{proof}

\bibliographystyle{plain}
\bibliography{article}

\begin{thebibliography}{10}

\bibitem{invariance}
David Alvarez-Melis, Stefanie Jegelka, and Tommi~S Jaakkola.
\newblock Towards optimal transport with global invariances.
\newblock In {\em International Conference on Artificial Intelligence and
  Statistics}, pages 1870--1879. PMLR, 2019.

\bibitem{quadratic}
Kurt Anstreicher and Henry Wolkowicz.
\newblock On {L}agrangian relaxation of quadratic matrix constraints.
\newblock In {\em Journal on Matrix Analysis and Applications}, volume~22,
  pages 41--55. SIAM, 2000.

\bibitem{arjovsky2017wasserstein}
Martin Arjovsky, Soumith Chintala, and L{\'e}on Bottou.
\newblock Wasserstein generative adversarial networks.
\newblock In {\em International Conference on Machine Learning}, pages
  214--223. PMLR, 2017.

\bibitem{bigot2017geodesic}
J{\'e}r{\'e}mie Bigot, Ra{\'u}l Gouet, Thierry Klein, Alfredo L{\'o}pez, et~al.
\newblock Geodesic {PCA} in the {W}asserstein space by convex {PCA}.
\newblock In {\em Annales de l'Institut Henri Poincar{\'e}, Probabilit{\'e}s et
  Statistiques}, volume~53, pages 1--26. Institut Henri Poincar{\'e}, 2017.

\bibitem{blanchet2019robust}
Jose Blanchet, Yang Kang, and Karthyek Murthy.
\newblock Robust {W}asserstein profile inference and applications to machine
  learning.
\newblock {\em Journal of Applied Probability}, 56(3):830--857, 2019.

\bibitem{brenier}
Yann Brenier.
\newblock Polar factorization and monotone rearrangement of vector-valued
  functions.
\newblock In {\em Communications on Pure and Applied Mathematics}, volume~44,
  pages 375--417. Wiley, 1991.

\bibitem{dimension}
Yuhang Cai and Lek-Heng Lim.
\newblock Distances between probability distributions of different dimensions.
\newblock In {\em arXiv preprint}, 2020.

\bibitem{reftheo4}
Samir Chowdhury and Tom Needham.
\newblock Gromov--{W}asserstein averaging in a {R}iemannian framework.
\newblock In {\em Conference on Computer Vision and Pattern Recognition
  Workshops}, pages 842--843. IEEE/CVF, 2020.

\bibitem{courty2016optimal}
Nicolas Courty, R{\'e}mi Flamary, Devis Tuia, and Alain Rakotomamonjy.
\newblock Optimal transport for domain adaptation.
\newblock In {\em Transactions on Pattern Analysis and Machine Intelligence},
  volume~39, pages 1853--1865. IEEE, 2016.

\bibitem{dowson1982frechet}
DC~Dowson and BV~Landau.
\newblock The {F}r{\'e}chet distance between multivariate normal distributions.
\newblock {\em Journal of multivariate analysis}, 12(3):450--455, 1982.

\bibitem{galichon2014stochastic}
Alfred Galichon, Pierre Henry-Labordere, Nizar Touzi, et~al.
\newblock A stochastic control approach to no-arbitrage bounds given marginals,
  with an application to lookback options.
\newblock {\em The Annals of Applied Probability}, 24(1):312--336, 2014.

\bibitem{genevay2017learning}
Aude Genevay, Gabriel Peyre, and Marco Cuturi.
\newblock Learning {G}enerative {M}odels with {S}inkhorn divergences.
\newblock In {\em International Conference on Artificial Intelligence and
  Statistics}, volume~84, pages 1608--1617. PMLR, 2018.

\bibitem{OTGaussian}
Clark~R Givens, Rae~Michael Shortt, et~al.
\newblock A class of {W}asserstein metrics for probability distributions.
\newblock In {\em Michigan Mathematical Journal}, volume~31, pages 231--240.
  the University of Michigan, 1984.

\bibitem{Isserlis}
Leon Isserlis.
\newblock On a formula for the product-moment coefficient of any order of a
  normal frequency distribution in any number of variables.
\newblock In {\em Biometrika}, volume~12, pages 134--139. JSTOR, 1918.

\bibitem{james1976topology}
Ioan~Mackenzie James.
\newblock {\em The topology of {S}tiefel manifolds}, volume~24.
\newblock Cambridge University Press, 1976.

\bibitem{memoli}
Facundo M{\'e}moli.
\newblock Gromov--{W}asserstein distances and the metric approach to object
  matching.
\newblock In {\em Foundations of Computational Mathematics}, volume~11, pages
  417--487. Springer, 2011.

\bibitem{temd}
Ofir Pele and Ben Taskar.
\newblock The tangent earth mover’s distance.
\newblock In {\em International Conference on Geometric Science of
  Information}, pages 397--404. Springer, 2013.

\bibitem{OT}
Gabriel Peyr{\'e}, Marco Cuturi, et~al.
\newblock Computational optimal transport: with applications to data science.
\newblock In {\em Foundations and Trends in Machine Learning}, volume~11, pages
  355--607. Now Publishers Inc., 2019.

\bibitem{algo}
Gabriel Peyr{\'e}, Marco Cuturi, and Justin Solomon.
\newblock Gromov-{W}asserstein averaging of kernel and distance matrices.
\newblock In {\em International Conference on Machine Learning}, pages
  2664--2672. PMLR, 2016.

\bibitem{color}
Julien Rabin, Sira Ferradans, and Nicolas Papadakis.
\newblock Adaptive color transfer with relaxed optimal transport.
\newblock In {\em International Conference on Image Processing}, pages
  4852--4856. IEEE, 2014.

\bibitem{textures}
Julien Rabin, Gabriel Peyr{\'e}, Julie Delon, and Marc Bernot.
\newblock Wasserstein barycenter and its application to texture mixing.
\newblock In {\em International Conference on Scale Space and Variational
  Methods in Computer Vision}, pages 435--446. Springer, 2011.

\bibitem{reftheo5}
Filippo Santambrogio.
\newblock Optimal transport for applied mathematicians.
\newblock In {\em Birk{\"a}user NY}, volume~55, page~94. Springer, 2015.

\bibitem{reftheo3}
Karl-Theodor Sturm.
\newblock The space of spaces: curvature bounds and gradient flows on the space
  of metric measure spaces.
\newblock In {\em arXiv preprint}, 2012.

\bibitem{takatsu2010wasserstein}
Asuka Takatsu.
\newblock On {W}asserstein geometry of {G}aussian measures.
\newblock In {\em Probabilistic approach to geometry}, pages 463--472.
  Mathematical Society of Japan, 2010.

\bibitem{These}
Titouan Vayer.
\newblock A contribution to optimal transport on incomparable spaces.
\newblock In {\em arXiv preprint}, 2020.

\bibitem{reftheo2}
C~Villani.
\newblock {\em Optimal transport: old and new}, volume 338.
\newblock Springer Science \& Business Media, 2008.

\bibitem{reftheo1}
C\'edric Villani.
\newblock {\em Topics in optimal transportation}.
\newblock Number~58. American Mathematical Soc., 2003.

\end{thebibliography}



 






\end{document}